\documentclass[10pt, a4paper]{amsart}
\usepackage[british]{babel}
\usepackage{amsmath}
\usepackage{amssymb}
\usepackage{amsthm}
\usepackage{graphicx}
\usepackage{url}
\usepackage{enumerate}
\usepackage{xcolor}
\usepackage{mathrsfs} 
\usepackage{t1enc}
\usepackage{mathtools}
\usepackage{placeins} 
\usepackage{tikz}
\usepackage{textcomp}
\usepackage{appendix}
\usetikzlibrary{matrix}
\usetikzlibrary{arrows,automata}
\usetikzlibrary{decorations.pathreplacing}
\usetikzlibrary{shapes.geometric}
\usepackage[colorlinks=true,allcolors=blue]{hyperref}

\newtheorem{thm}{Theorem}[section]
\newtheorem{lem}[thm]{Lemma}
\newtheorem{prop}[thm]{Proposition}
\newtheorem{obs}[thm]{Observation}
\newtheorem{cor}[thm]{Corollary}
\newtheorem{clm}[thm]{Claim}

\theoremstyle{definition}
\newtheorem{defn}[thm]{Definition}
\newtheorem{rem}[thm]{Remark}

\DeclareMathOperator{\Int}{int}

\DeclareMathOperator{\diam}{diam}

\DeclareMathOperator{\scr}{cr}

\DeclareMathOperator{\Rot}{Rot}

\newcommand{\eps}{\varepsilon}
\newcommand{\R}{\mathbb{R}}
\newcommand{\Z}{\mathbb{Z}}
\newcommand{\N}{\mathbb{N}}

\newcommand{\red}{\color{red}}

\newcommand{\Ci}{\mathbb{S}^1}
\newcommand{\B}{\mathcal{B}}
\newcommand{\degr}{\mathrm{deg}}

\title[Parameterized family of pseudo-circle attractors]{Parameterized family of annular homeomorphisms with pseudo-circle attractors}
\begin{document}
\author{Jernej \v Cin\v c}
\address[J. \v Cin\v c]{
	University of Maribor, Koro\v ska 160, 2000 Maribor, Slovenia
	-- and --
	AGH University of Krakow, Faculty of Applied Mathematics,
	al. Mickiewicza 30,
	30-059 Krak\'ow,
	Poland  
	-- and -- National Supercomputing Centre IT4Innovations, University of Ostrava,
	IRAFM,
	30. dubna 22, 70103 Ostrava,
	Czech Republic}
\email{jernej.cinc@um.si}

\author{Piotr Oprocha}
\address[P. Oprocha]{
	AGH University of Krakow, Faculty of Applied Mathematics,
	al. Mickiewicza 30,
	30-059 Krak\'ow,
	Poland -- and -- National Supercomputing Centre IT4Innovations, University of Ostrava,
	IRAFM,
	30. dubna 22, 70103 Ostrava,
	Czech Republic}
\email{oprocha@agh.edu.pl}
\begin{abstract}
	In this paper we construct a paramaterized family of annular homeomorphisms with Birkhoff-like rotational attractors that vary continuously with the parameter, are all homeomorphic to the pseudo-circle, display interesting boundary dynamics and furthermore preserve the induced Lebesgue measure from the circle. Namely, in the constructed family of attractors the outer prime ends rotation number vary continuously with the parameter through the interval $[0,1/2]$. This, in particular, answers a question from [J. London Math. Soc. (2) {\bf 102} (2020), 557--579]. To show main results of the paper we first prove a result of an independent interest, that Lebesgue-measure preserving circle maps generically satisfy the crookedness condition which implies that generically the inverse limits of Lebesgue measure-preserving circle maps are hereditarily indecomposable. For degree one circle maps, this implies that the generic inverse limit in this context is the R.H. Bing's pseudo-circle.
\end{abstract}

\maketitle

\section{Introduction}

The main goal of this paper is to study a peculiar parameterized family of strange Birkhoff-like annular homeomorphisms with topologically unique rotational attractors\footnote{Following \cite{KorPas} a {\em rotational attractor} is a topological attractor for which the corresponding external prime ends rotation number is non-zero (mod 1).} with pairwise different boundary dynamics which in addition preserve the induced Lebesgue measure from the circle.

To obtain such a family we first prove a result of an independent interest, that Lebesgue-measure preserving circle maps $C_{\lambda}(\Ci)$ generically satisfy the crookedness condition which implies that generically the inverse limits of Lebesgue measure-preserving circle maps are the pseudo-solenoids. For degree one circle maps this implies that the generic inverse limit in this context is the R.H. Bing's pseudo-circle. Pseudo-circle is a {\em continuum} (compact connected metric space) of much interest in Continuum Theory. Bing proved in \cite{BingPacific}, it is {\em hereditarily indecomposable}, meaning that none of its subcontinua can be expressed as a union of two of its proper subcontinua. Since Handel's construction \cite{Handel}, where he realized the pseudo-circle as a minimal set of a $C^{\infty}$ planar diffeomorphisms, it appeared in Dynamical Systems several times (see the introduction of \cite{BCO} for extended information on this topic), most recently Boro\'nski, Kennedy, Liu and Oprocha proved that the pseudo-circle admits a minimal non-invertible self-map \cite{BLKO}.
Generic properties of Lebesgue measure-preserving endomorphisms on one-dimensional manifolds have been recently studied in detail following the appearance of the paper of Bobok and Troubetzkoy \cite{BT}.
In \cite{BCOT1} Bobok, \v Cin\v c, Oprocha and Troubetzkoy showed that s-limit shadowing, a very strong version of the shadowing property, is a generic property in $C_{\lambda}(\Ci)$. Recently, the same authors studied in \cite{BCOT2} generic topological and measure-theoretic expansion properties of maps from $C_{\lambda}(\Ci)$ and proved that locally eventually onto (leo) condition is satisfied on an open dense subset of maps from $C_{\lambda}(\Ci)$ and that measure-theoretic weak mixing is a generic property. For the review of other generic results in $C_{\lambda}(\Ci)$ as well as the connection with the spaces of circle maps preserving any other non-atomic measure with support $\Ci$ and spaces of circle maps with dense periodicity see the recent survey paper \cite{BCOT3}.
As we already announced, the first main result of the paper is the following.

\begin{thm}\label{thm:UniLimPresLeb}
The set $\mathcal{T}$ of all maps from $C_{\lambda}(\mathbb{S}^1)$ such that, if $f\in \mathcal{T}$ then
for every $\delta>0$ there exists a positive integer $n$ so that $f^n$ is $\delta$-crooked, contains a dense $G_{\delta}$ subset of $C_{\lambda}(\mathbb{S}^1)$.
\end{thm}

The result analogous to Theorem~\ref{thm:UniLimPresLeb} and the following Corollary~\ref{cor:MincTransue} holds also in the setting of the closure of maps with dense set of periodic points, analogously as in \cite{CO}.

\begin{cor}\label{cor:MincTransue}
The inverse limit with any $C_{\lambda}(\Ci)$-generic map as a single bonding map is hereditarily indecomposable. In particular, for the subset of degree $1$ circle maps, the inverse limit with any $C_{\lambda}(\Ci)$-generic  map as a single bonding map is the pseudo-circle.
\end{cor}

We obtain the required annular homeomorphisms through a very useful technique called BBM (Brown-Barge-Martin) embeddings. 
BBM embedding is a technique that allows one to produce a parameterized family of strange attractors from a (possibly very complicated) map acting on a boundary retract of a manifold with boundary.  The basic objects in the BBM construction are inverse limits and natural extensions of the underlying bonding maps. The original idea goes back to the paper of Barge and Martin \cite{BM}, where the authors constructed strange attractors from a wide class of inverse limits. One of the crucial steps for this technique to work is the usage of Brown's approximation theorem \cite{Bro}.
A parameterized version of BBM embeddings was described by Boyland, de Carvalho and Hall in \cite{3G-BM} and boundary dynamics of BBM embeddings of tent inverse limits was studied in detail by the same authors in \cite{BdCH} and subsequently by Anu\v si\'c and \v Cin\v c \cite{AC}. A measure theoretic aspect of these embeddings was studied in \cite{BdCH1}.
Subsequently, BBM technique was used to construct new parameterized families of strange attractors in \cite{BCL} and \cite{CO}. In \cite{BCL}, Boro\'nski, \v Cin\v c and Liu asked in Question 2 if there exist a parametrized family of sphere homeomorphisms such that for any parameter, homeomorphisms preserve a cofrontier attractor which is homeomorphic to the pseudo-circle. Note that our
Theorem~\ref{lem:BBM1} stated below, in particular, answers this question in the affirmative. Furthermore, our constructed family of homeomorphisms also has interesting topological and measure-theoretic expansion properties.

The following definition will help the reader to understand certain not so standard measure-theoretic parts of the main theorem of this paper.

\begin{defn}\label{def:induced} Let $X$ be a Euclidean space with Lebesgue measure $\lambda$ and let $f\colon X\to X$ be a (surjective) map. An invariant measure $\hat{\mu}_f$ for the natural extension $\hat f\colon \hat X\to \hat X$ is called the \emph{inverse limit physical measure} if $\hat{\mu}_f$ has a basin $\hat B$ so that  $\lambda(\pi_0(\hat B))>0$.
\end{defn}

\begin{thm}\label{lem:BBM1}
Let $\mathcal{T}$ be a dense $G_{\delta}$ subset of $C_{\lambda}(\Ci)$ from Theorem~\ref{thm:UniLimPresLeb}.
 There is a parametrized family of circle maps $\{f_t\}_{t\in [0,1]}\subset \mathcal{T}\subset C_{\lambda}(\Ci)$  so that every $f_t$ is weakly mixing with respect to $\lambda$ and a parametrized family of annular homeomorphisms $\{\Phi_{t}\}_{t\in[0,1]}\subset \mathcal{H}(\mathbb{A}, \mathbb{A})$ varying continuously with $t$ having $\Phi_t$-invariant  pseudo-circle cofrontier Birkhoff-like attractors $\Lambda_t\subset \mathbb{A}$ for every $t\in [0,1]$ so that 
	\begin{itemize}
		\item[(a)]\label{item(a)}  $\Phi_t|_{\Lambda_t}$ is topologically conjugate to the natural extension $\hat{f}_{t}\colon \hat{\mathbb{S}}^1\to \hat{\mathbb{S}}^1$.
		\item[(b)]  The inverse limit measure $\hat{\mu}_t$ induced by $\lambda$ is (up to an isomorphism) inverse limit physical on $\mathbb{A}$ and $\Phi_t$ is weakly mixing with respect to $\hat{\mu}_t$.
   Measures $\hat{\mu}_t$ vary continuously in weak* topology.
  \item[(c)]\label{item(b)} The attractors $\{\Lambda_t\}_{t\in [0,1]}$ vary continuously in the Hausdorff metric.  
		\item[(d)]\label{item(c)} The outer prime ends rotation numbers of homeomorphisms $\Phi_{t}$ vary continuously with $t$ in the interval $[0,1/2]$. $\Lambda_t$ are rotational attractors for all $t\in (0,1/2]$.
		\item[(e)] There are uncountably many dynamically non-equivalent planar embeddings of the pseudo-circle in the family $\{\Lambda_t\}_{t\in[0,1]}$.
		\end{itemize}
\end{thm}

 Our constructed family of homeomorphisms in Theorem~\ref{lem:BBM1} is in annulus but both boundary components are closed and invariant under the homeomorphisms, thus by collapsing each of the components points we can get a family of sphere homeomorphisms with analogous properties as in the theorem.

Since the attractors in the parameterized family are pairwise homeomorphic to each other, the main difficulty of the construction lies in determining that the embeddings support a non-conjugated dynamics. 
Recently, a similar result \cite{CO} was obtained starting from interval maps and constructing a family of planar homeomorphisms with pseudo-arc attractors so that the prime ends rotation numbers vary continuously with the parameter on the interval $[0,1/2]$. The construction in \cite{CO} is, however, much simpler because of the presence of two endpoints 0 and 1 of $I$, see Figure~\ref{fig:interval}. There, we attach arcs that we keep accessible through the inverse limit procedure and we either make them fixed (prime end rotation $0$) or permute them, which gives the prime end rotation number of the attractor $1/2$.
If we work with circle maps there is no clear way how to ensure that a point which is accessible in some level of the construction stays accessible after perturbing the map, see Figure~\ref{fig:circle}.

Our approach to deal with this difficulty is as follows. To maintain the periodicity of selected points we apply small rotations after perturbations. This procedure is delicate and needs to be done with special care since one can revert desired properties achieved by (and needed from) the perturbations.
Moreover, in the process of the proof, one of the main steps is to ensure that the associated "water pouring maps" do not change image of some points, however the original maps remain sufficiently crooked.
Furthermore, we need to apply both operations (rotations and perturbations) to the whole family of maps simultaneously which makes the procedure even more delicate.
To make certain points accessible we, in addition, need to define the unwrappings in the BBM construction carefully, see Figure~\ref{fig:lifting}.

Our article is structured as follows. In Section~\ref{sec:pre} we give general preliminaries needed in the basic setting of our article. In Section~\ref{sec:thm1.1} we prove Theorem~\ref{thm:UniLimPresLeb}. The proof has a similar flavour as the proof of Theorem 1.1. from \cite{CO} however some delicate modifications are needed, therefore these different parts of the proofs are written out completely. Most importantly, several changes are required in the proof of Lemma~\ref{lem:MincUpdt} compared to \cite{CO}. The rest of the article (Section \ref{sec:family}) is dedicated to proving Theorem~\ref{thm:main}.

\begin{figure}
\begin{tikzpicture}[scale=0.8]
\draw[solid](0,0) circle (2);
\draw[ultra thick](-1,0)--(1,0);
\draw[ultra thick](-1,0)--(1,0);
\node at (0.3,-0.5) {\small $I$};
\draw[teal,thick] (1,-0.2)--(0,-0.2)--(1/2,-0.1)--(-1/2,-0.1)--(1/2,0)--(-1/2,0.1)--(0,0.2)--(-1,0.2);
\node[circle,fill, inner sep=1.5] at (-1,0){};
\node[circle,fill, inner sep=1.5] at (1,0){};
\node at (-0.3,0.5) {\small {\color{teal} $f$}};
\draw[red](-1,0)--(-2,0);
\draw[red](1,0)--(2,0);
\end{tikzpicture}
\caption{In the interval case, if we use the BBM construction, there is a an intuitively simple procedure that allows the endpoints of the interval fixed and accessible provided they are fixed points or of period two of the acting unwrapping of a map $f$. Such a procedure was used e.g. in \cite{BCL} as well as in \cite{CO}.   Lack of endpoints disables its application for the circle.}\label{fig:interval}
\end{figure}
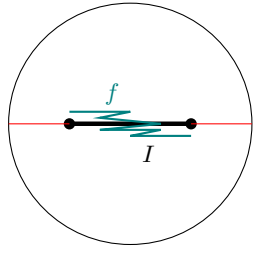

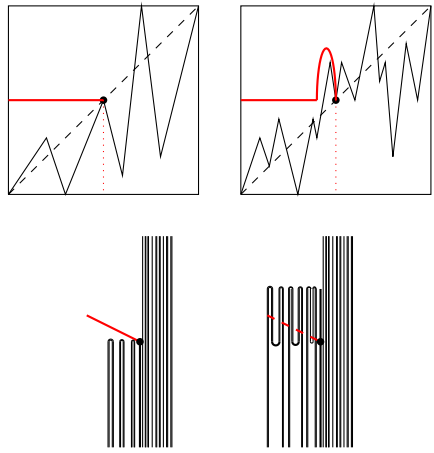
\begin{figure}
\begin{tikzpicture}[scale=2.5]
\draw[dashed] (0,0)--(1,1);
\draw (0,0)--(1,0)--(1,1)--(0,1)--(0,0);
\draw (0,0)--(0.2,0.3)--(0.3,0)--(0.5,0.5)--(0.6,0.1)--(0.7,1)--(0.8,0.2)--(1,1);
\node[circle,fill, inner sep=1] at (0.5,0.5){};
\draw[red,thick] (0,0.5)--(0.5,0.5);
\draw[red,dotted] (0.5,0.5)--(0.5,0);
\end{tikzpicture}
\hspace{0.3cm}
\begin{tikzpicture}[scale=2.5]
\draw (0,0)--(1,0)--(1,1)--(0,1)--(0,0);
\draw[dashed] (0,0)--(1,1);
\draw (0,0)--(0.1,0.3)--(0.15,0.15)--(0.2,0.4)--(0.3,0)--(0.38,0.4)--(0.4,0.3)--(0.47,0.7)--(0.5,0.5)--(0.53,0.7)--(0.6,0.5)--(0.7,1)--(0.73,0.6)--(0.76,0.7)--(0.8,0.2)--(0.87,0.8)--(0.93,0.5)--(1,1);
\node[circle,fill, inner sep=1] at (0.5,0.5){};
\draw[red,thick] (0,0.5)--(0.4,0.5);
\draw[domain=0:180,red,thick] plot ({0.45+0.05*cos(\x)}, {0.5+0.275*sin(\x)});
\draw[red,dotted] (0.5,0.5)--(0.5,0);
\end{tikzpicture}

\vspace{0.5cm}
\begin{tikzpicture}[scale=0.7,rotate=90]
		\draw (5,1.6)--(7,1.6);
		\draw (5,1.5)--(7,1.5);
		\draw (5,1.58)--(7,1.58);
		\draw (5,1.52)--(7,1.52);
		\draw (5,1.38)--(7,1.38);
		\draw (5,1.3)--(7,1.3);
		\draw (5,1.365)--(7,1.365);
		\draw (5,1.315)--(7,1.315);
		\draw (5,1.16)--(7,1.16);
		\draw (5,1.1)--(7,1.1);
		\draw (5,1.15)--(7,1.15);
		\draw (5,1.11)--(7,1.11);
		\draw (5,0.4)--(9,0.4);
		\draw (5,0.5)--(9,0.5);
		\draw (5,0.42)--(9,0.42);
		\draw (5,0.48)--(9,0.48);
		\draw (5,0.62)--(9,0.62);
		\draw (5,0.7)--(9,0.7);
		\draw (5,0.635)--(9,0.635);
		\draw (5,0.685)--(9,0.685);
		\draw (5,0.84)--(9,0.84);
		\draw (5,0.9)--(9,0.9);
		\draw (5,0.85)--(9,0.85);
		\draw (5,0.89)--(9,0.89);
		\draw[red,thick] (7.5,2)--(7,1);
		\draw[domain=270:450] plot ({7+0.05*cos(\x)}, {1.55+0.05*sin(\x)});
		\draw[domain=270:450] plot ({7+0.03*cos(\x)}, {1.55+0.03*sin(\x)});
		\draw[domain=270:450] plot ({7+0.04*cos(\x)}, {1.34+0.04*sin(\x)});
		\draw[domain=270:450] plot ({7+0.025*cos(\x)}, {1.34+0.025*sin(\x)});
		\draw[domain=270:450] plot ({7+0.03*cos(\x)}, {1.13+0.03*sin(\x)});
		\draw[domain=270:450] plot ({7+0.02*cos(\x)}, {1.13+0.02*sin(\x)});
		\draw[thick] (5, 1)--(7,1);
		\node[circle,fill, inner sep=1] at (7,1){};
		\draw (5,0.95)--(9,0.95);
		\draw (5,0.77)--(9,0.77);
		\draw (5,0.56)--(9,0.56);
		\end{tikzpicture}
		\hspace{1cm}
		\begin{tikzpicture}[scale=0.7,rotate=90]
		\draw (5,1.6)--(8,1.6);
		\draw (7,1.52)--(8,1.52);
		\draw (5,1.58)--(8,1.58);
		\draw (7,1.54)--(8,1.54);
		\draw (7,1.4)--(8,1.4);
		\draw (5,1.34)--(8,1.34);
		\draw (7,1.42)--(8,1.42);
		\draw (5,1.355)--(8,1.355);
		\draw[domain=270:450] plot ({8+0.04*cos(\x)}, {1.56+0.04*sin(\x)});
		\draw[domain=270:450] plot ({8+0.03*cos(\x)}, {1.56+0.03*sin(\x)});
		\draw[domain=270:450] plot ({8+0.04*cos(\x)}, {1.38+0.04*sin(\x)});
		\draw[domain=270:450] plot ({8+0.025*cos(\x)}, {1.38+0.025*sin(\x)});
		\draw[domain=90:270,thick] plot ({7+0.06*cos(\x)}, {1.47+0.06*sin(\x)});
		\draw (5,2)--(8,2);
		\draw (7,1.9)--(8,1.9);
		\draw (5,1.98)--(8,1.98);
		\draw (7,1.92)--(8,1.92);
		\draw (7,1.78)--(8,1.78);
		\draw (5,1.7)--(8,1.7);
		\draw (7,1.765)--(8,1.765);
		\draw (5,1.715)--(8,1.715);
		\draw[domain=270:450] plot ({8+0.05*cos(\x)}, {1.95+0.05*sin(\x)});
		\draw[domain=270:450] plot ({8+0.03*cos(\x)}, {1.95+0.03*sin(\x)});
		\draw[domain=270:450] plot ({8+0.04*cos(\x)}, {1.74+0.04*sin(\x)});
		\draw[domain=270:450] plot ({8+0.025*cos(\x)}, {1.74+0.025*sin(\x)});
		\draw[domain=90:270,thick] plot ({7+0.069*cos(\x)}, {1.841+0.069*sin(\x)});
		\draw (5,1.26)--(8,1.26);
		\draw (7,1.19)--(8,1.19);
		\draw (5,1.24)--(8,1.24);
		\draw (7,1.14)--(8,1.14);
		\draw (5,1.10)--(8,1.10);
		\draw (5,1.08)--(8,1.08);
		\draw[domain=270:450] plot ({8+0.04*cos(\x)}, {1.22+0.04*sin(\x)});
		\draw[domain=270:450] plot ({8+0.03*cos(\x)}, {1.215+0.03*sin(\x)});
		\draw[domain=270:450] plot ({8+0.04*cos(\x)}, {1.115+0.03*sin(\x)});
		\draw[domain=270:450] plot ({8+0.025*cos(\x)}, {1.108+0.02*sin(\x)});
		\draw[domain=90:270] plot ({7+0.025*cos(\x)}, {1.165+0.025*sin(\x)});
		\draw (5,0.4)--(9,0.4);
		\draw (5,0.5)--(9,0.5);
		\draw (5,0.42)--(9,0.42);
		\draw (5,0.48)--(9,0.48);
		\draw (5,0.62)--(9,0.62);
		\draw (5,0.7)--(9,0.7);
		\draw (5,0.635)--(9,0.635);
		\draw (5,0.685)--(9,0.685);
		\draw (5,0.84)--(9,0.84);
		\draw (5,0.9)--(9,0.9);
		\draw (5,0.85)--(9,0.85);
		\draw (5,0.89)--(9,0.89);
		\draw[red,thick,dashed] (7.5,2)--(7,1);
		\draw[thick] (5, 1)--(8,1);
		\node[circle,fill, inner sep=1] at (7,1){};
		\draw (5,0.95)--(9,0.95);
		\draw (5,0.77)--(9,0.77);
		\draw (5,0.56)--(9,0.56);
		\end{tikzpicture}
\caption{In the upper row on the left there is a picture of a circle map (represented on the interval) and we denote a certain value of local extremum that we want to make accessible. The right upper picture represent this maps after a perturbation. In the lower row we  depicted how such a point may be positioned in the inverse limit. In the
 continuum on the left
it can indeed be accessed by a red arc but in the right  one it cannot be accessed any more. 
}\label{fig:circle}
\end{figure}

\section{Preliminaries}\label{sec:pre}

Let $\mathbb{N}:=\{1,2,3,\ldots\}$ and $\mathbb{N}_0:=\mathbb{N}\cup\{0\}$, let $I:=[0,1]\subset \mathbb{R}$ denote the unit interval and $\mathbb{S}^1$ the unit circle.  For $x,y\in\mathbb{S}^1$ let $d(x,y)$ denote the minimal normalized arc-length distance on $\Ci$ between $x$ and $y$. 

Consider a continuous map $f\colon~\mathbb S^1\to \mathbb S^1$ of
\emph{degree} $\degr(f)\in\Z$. Let $\tilde f\colon~\R\to\R$ be a \emph{lifting} of $f$, i.e., a continuous map for which
\begin{equation}\label{e:10}\phi\circ \tilde f = f\circ\phi\text{ on }\R,
\end{equation}
where $\phi\colon~\R\to\Ci$ is defined by $\phi(x)= e^{2\pi ix}$. Then $\tilde f(x + 1) = \tilde f(x) + \mathrm{deg}(f)$ for each $x\in\R$.  
If $F=\tilde f\vert [0,1)\pmod 1$, we say that $F\colon~[0,1)\to [0,1)$ is a \emph{representative} of $f$. Note that since two liftings of $f$ differ by an integer constant, $F$ does not depend on a concrete choice of a lifting of $f$.
Note that $f$ is onto if and only if its representative $F=\tilde{f}| [0,1)\pmod 1$ is onto.

Let  $\lambda$ denote the \emph{normalized Lebesgue measure} on $\Ci$ and $\B$ the family of all \emph{Borel sets} in $\Ci$. In this paper we will work with \emph{continuous maps from $\Ci$ into $\Ci$ preserving $\lambda$}, which we denote by
$$ C_{\lambda}(\Ci):=\{f\colon~\Ci\to \Ci\colon~\forall
A\in\B,~\lambda(A)=\lambda(f^{-1}(A))\}. $$
We consider the set $C_{\lambda}(\Ci)$ equipped with the \emph{uniform metric} $\rho$:
$$\rho (f,g) := \sup_{x \in \Ci} |f(x) - g(x)|.$$ 
Space $(C_{\lambda}(\Ci),\rho)$  is a complete metric space, see \cite[Lemma 2.5]{BCOT1}.

The next lemma describes the representatives of Lebesgue measure-preserving circle maps and its proof can be found in \cite{BCOT1}. By abuse of notation we denote by $\lambda$ the  Lebesgue measure on $\mathbb{R}$ as well.

\begin{lem}\label{l:3} Let $F$ be a representative of $f\colon~\Ci\to\Ci$. The following conditions are equivalent.
\begin{itemize}
\item[(i)] $f\in C_{\lambda}(\Ci)$.
\item[(ii)] $\forall A\subset [0,1)\text{ Borel },~\lambda(A)=\lambda(F^{-1}(A))$.
    \end{itemize}
 \end{lem}

A \emph{critical point} $x\in \mathbb{S}^1$ of $f:C(\Ci)\to C(\Ci)$ is a point such that $f|_{J}$ is not one-to-one for every open arc $x\in J\subset \mathbb{S}^1$. 
An arc $J\subset \mathbb{S}^1$ is called an {\em arc of monotonicity} if $f|_{A}$ is monotone, but $f$ is not monotone on any arc properly containing $J$.
We say that a map $f:\mathbb{S}^1\to \mathbb{S}^1$ is \emph{piecewise linear} (or \emph{piecewise linear}) if it has finitely many critical points and is affine on every arc of monotonicity. We will say that $\tilde f$, a lifting of $f$ is \emph{piecewise affine} if its corresponding representative $F$ is piecewise affine. We say that a circle map $f$ is \emph{locally eventually onto (leo)} if for every open arc $J\subset \mathbb{S}^1$ there exists a non-negative integer $n$ so that $f^n(J)=\mathbb{S}^1$. This property is also sometimes referred in the literature as \emph{topological exactness}.

To check whether our perturbations preserve Lebesgue measure we will also implicitly use the following simple observation which we can apply for the representatives.

\begin{lem}\label{l:6}
Let $F$ be a piecewise affine representative of a circle map $f$ with non-zero slopes and such that its derivative does not exist at a finite set $E$. Then the properties (i) and (ii) from Lemma~\ref{l:3} are equivalent to the property
\begin{equation}\label{e:3}
   \forall~y\in [0,1)\setminus F(E)\colon~\sum_{x\in F^{-1}(y)}\frac{1}{\vert F'(x)\vert}=1.
 \end{equation}
 \end{lem}

For a metric space $(X,d)$ we shall use $B(x,\xi)$ for the open ball of radius $\xi$ centered at $x\in X$ and for a set $U\subset X$ we shall denote 
$$B(U,\xi):=\bigcup_{x\in U}B(x,\xi).$$ In the rest of the paper we use the letter $d$ to denote the \emph{Euclidean distance} on the underlying Euclidean space.

\section{Proof of Theorem~\ref{thm:UniLimPresLeb}}\label{sec:thm1.1}
For the easier visualization of the concept defined in the following definition we refer the reader to Figure~\ref{fig:sigma}, where the examples of such interval maps are given and Figure~\ref{fig:hatlambda} and \ref{fig:lambda} where such a circle map is given.

First let us give a definition of crookedness for the circle.

\begin{defn}\label{def:KTT}
Let $f:\mathbb{S}^1\to \mathbb{S}^1$ be a circle map with the usual arc-length metric $\rho$. For some $\delta>0$ a map $\omega:[0,1]\to \mathbb{S}^1$ is said to be {\em $(f,\delta)$-crooked} if there exist $s,t\in \mathbb{S}^1$ with $0<s\leq t<1$ such that $\rho(f\circ \omega(s),f\circ \omega(1))\leq \delta$ and $\rho(f\circ \omega(t),f\circ \omega(0))\leq \delta$.
\end{defn}

The following definition is a rephrasing of Definition~\ref{def:KTT} in the circle case in terms of liftings.

\begin{defn}\label{def:crooked}
    Let $f:\mathbb{S}^1\to \mathbb{S}^1$ be a circle map and $\tilde{f}:\mathbb{R}\to \mathbb{R}$ its lifting.
    Let $a,b\in \mathbb{R}$ and let $\delta>0$.
    We say that a map $f$ is {\em $\delta$-crooked between points $a$ and $b$}  if for every two points $c, d \in \mathbb{R}$ with $\tilde f(c) = a$ and $\tilde f(d) = b$, there is a point $c'$ between $c$ and $d$ and a point $d'$ between $c'$ and $d$ such that $|b - \tilde f(c')|\pmod 1 < \delta$ and $|a - \tilde f(d')|\pmod 1 < \delta$. We will say that $f$ is {\em $\delta$-crooked} if it is $\delta$-crooked between every pair of points. Given $\delta>0$ we say that $f$ is {\em $\delta$-crooked} if every map $\omega$ is $(f,\delta)$-crooked.
    
\end{defn}

\begin{rem}
It is clear that the definitions of $\delta$-crooked maps from Definitions~\ref{def:KTT} and \ref{def:crooked} are equivalent for the case of $\mathbb{S}^1$.
\end{rem}

The notion of $\delta$-crookedness goes hand in hand with inverse limits that we introduce now.
 For a collection of continuous maps on compact metric spaces $f_i:Z_{i+1}\to Z_i$ we define
\begin{equation}
\underleftarrow{\lim} (Z_i,f_i)
:=
\{\hat z:=\big(z_{0},z_1,\ldots \big) \in Z_0\times Z_1,\ldots\big|  
z_i\in Z_i, z_i=f_i(z_{i+1}), \forall i \geq 0\}.
\end{equation}
The space $\underleftarrow{\lim} (Z_i,f_i)$ is equipped with the subspace 
metric induced from the 
\emph{product metric} in $Z_0\times Z_1\times\ldots$,
where $f_i$ are called the {\em bonding maps}.
If $Z_i=Z$ and $f_i=f$ for all $i\geq 0$, the inverse limit space 
$$
\hat Z:=\underleftarrow{\lim} (Z,f)
$$
also comes with a natural homeomorphism, 
called the \emph{natural extension} of $f$ (or the 
\emph{shift homeomorphism})
$\hat f:\hat Z
\to \hat Z$, 
defined as follows. 
For any $\hat z= \big(z_{0},z_1,\ldots \big)\in \hat Z$,
\begin{equation}
\hat{f}(\hat z):= \big(f(z_0),f(z_{1}),f(z_2),\ldots \big) =\big(f(z_0),z_{0},z_1,\ldots \big).
\end{equation}
By $\pi_{i}$ we shall denote
the \emph{$i$-th projection} 
from 
$\hat Z$ to its $i$-th coordinate. 

 Note that {\em pseudo-circle} is the unique plane-separating circle-like continuum that is hereditarily indecomposable \cite{Fearnley}.
 Thus, from \cite[Theorem 3.3]{Fearnley2} (see also Theorem 23 in \cite{KOT} and comments afterwards) we get the following.

\begin{prop}\label{prop:pseudo-circle}
    Let $f:\mathbb{S}^1\to \mathbb{S}^1$ be a degree 1 circle map so that for every $\delta>0$ there exists $n\in \N$ such that $f^n$ is $\delta$-crooked. Then, $\underleftarrow{\lim}(\mathbb{S}^1,f)$ is the pseudo-circle.
\end{prop}

The following two definitions and the first part of the third definition were introduced in \cite{LM} and used in a similar fashion in \cite{CO}.
We will use the maps defined below as the building blocks of our perturbations. We will first work with interval maps and then lift them to $\mathbb{R}$ to define circle maps. The first part of the proof of Theorem~\ref{thm:UniLimPresLeb} starts analogously to the proof of Theorem 1.1 in \cite{CO} so we will just go briefly through its most important ingredients and often refer the reader to \cite{CO}.

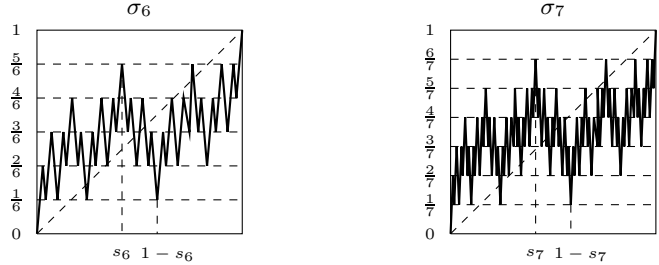
\begin{figure}[!ht]
	\centering
	\begin{tikzpicture}[scale=2.7]
	\draw[dashed] (0,0)--(1,1);
	\node at (0.5,1.1) {\small $\sigma_6$};
	\node at (-0.1,0) {\tiny $0$};
	\node at (-0.1,1/6) {\tiny $\frac{1}{6}$};
	\node at (-0.1,2/6) {\tiny $\frac{2}{6}$};
	\node at (-0.1,3/6) {\tiny $\frac{3}{6}$};
	\node at (-0.1,4/6) {\tiny $\frac{4}{6}$};
	\node at (-0.1,5/6) {\tiny $\frac{5}{6}$};
	\node at (-0.1,1) {\tiny $1$};
	\node at (29/70,-0.1) {\tiny $s_6$};
	\node at (41/70+0.05,-0.1) {\tiny $1-s_6$};
	\draw (0,0)--(0,1)--(1,1)--(1,0)--(0,0);
	\draw[dashed] (0,1/6)--(1,1/6);
	\draw[dashed] (0,2/6)--(1,2/6);
	\draw[dashed] (0,3/6)--(1,3/6);
	\draw[dashed] (0,4/6)--(1,4/6);
	\draw[dashed] (0,5/6)--(1,5/6);
	\draw[thick] (0,0)--(2/70,2/6)--(3/70,1/6)--(5/70,3/6)--(7/70,1/6)--(9/70,3/6)--(10/70,2/6)--(12/70,4/6)--(14/70,2/6)--(15/70,3/6)--(17/70,1/6)--(19/70,3/6)--(20/70,2/6)--(22/70,4/6)--(24/70,2/6)--(26/70,4/6)--(27/70,3/6)--(29/70,5/6)--(31/70,3/6)--(32/70,4/6)--(34/70,2/6)--(36/70,4/6)--(38/70,2/6)--(39/70,3/6)--(41/70,1/6)--(43/70,3/6)--(44/70,2/6)--(46/70,4/6)--(48/70,2/6)--(50/70,4/6)--(52/70,3/6)--(53/70,5/6)--(55/70,3/6)--(56/70,4/6)--(58/70,2/6)--(60/70,4/6)--(61/70,3/6)--(63/70,5/6)--(65/70,3/6)--(67/70,5/6)--(68/70,4/6)--(1,1);
	\draw[dashed] (29/70,5/6)--(29/70,0);
	\draw[dashed] (41/70,1/6)--(41/70,0);
	\end{tikzpicture}
	\hspace{2cm}
	\begin{tikzpicture}[scale=2.7]
	\draw[dashed] (0,0)--(1,1);
	\node at (0.5,1.1) {\small $\sigma_7$};
	\node at (-0.1,0) {\tiny $0$};
	\node at (-0.1,1/7) {\tiny $\frac{1}{7}$};
	\node at (-0.1,2/7) {\tiny $\frac{2}{7}$};
	\node at (-0.1,3/7) {\tiny $\frac{3}{7}$};
	\node at (-0.1,4/7) {\tiny $\frac{4}{7}$};
	\node at (-0.1,5/7) {\tiny $\frac{5}{7}$};
	\node at (-0.1,6/7) {\tiny $\frac{6}{7}$};
	\node at (-0.1,1) {\tiny $1$};
	\node at (70/169,-0.1) {\tiny $s_7$};
	\node at (99/169+0.05,-0.1) {\tiny $1-s_7$};
	\draw (0,0)--(0,1)--(1,1)--(1,0)--(0,0);
	\draw[dashed] (0,1/7)--(1,1/7);
	\draw[dashed] (0,2/7)--(1,2/7);
	\draw[dashed] (0,3/7)--(1,3/7);
	\draw[dashed] (0,4/7)--(1,4/7);
	\draw[dashed] (0,5/7)--(1,5/7);
	\draw[dashed] (0,6/7)--(1,6/7);
	\draw[thick] (0,0)--(2/169,2/7)--(3/169,1/7)--(5/169,3/7)--(7/169,1/7)--(9/169,3/7)--(10/169,2/7)--(12/169,4/7)--(14/169,2/7)--(15/169,3/7)--(17/169,1/7)--(19/169,3/7)--(20/169,2/7)--(22/169,4/7)--(24/169,2/7)--(26/169,4/7)--(27/169,3/7)--(29/169,5/7)--(31/169,3/7)--(32/169,4/7)--(34/169,2/7)--(36/169,4/7)--(38/169,2/7)--(39/169,3/7)--(41/169,1/7)--(43/169,3/7)--(44/169,2/7)--(46/169,4/7)--(48/169,2/7)--(50/169,4/7)--(52/169,3/7)--(53/169,5/7)--(55/169,3/7)--(56/169,4/7)--(58/169,2/7)--(60/169,4/7)--(61/169,3/7)--(63/169,5/7)--(65/169,3/7)--(67/169,5/7)--(68/169,4/7)--(70/169,6/7)--(72/169,4/7)--(73/169,5/7)--(75/169,3/7)--(77/169,5/7)--(79/169,3/7)--(80/169,4/7)--(82/169,2/7)--(84/169,4/7)--(85/169,3/7)--(87/169,5/7)--(89/169,3/7)--(90/169,4/7)--(92/169,2/7)--(94/169,4/7)--(96/169,2/7)--(97/169,3/7)--(99/169,1/7)--(101/169,3/7)--(102/169,2/7)--(104/169,4/7)--(106/169,2/7)--(108/169,4/7)--(109/169,3/7)--(111/169,5/7)--(113/169,3/7)--(114/169,4/7)--(116/169,2/7)--(118/169,4/7)--(121/169,3/7)--(121/169,5/7)--(123/169,3/7)--(125/169,5/7)--(126/169,4/7)--(128/169,6/7)--(130/169,4/7)--(131/169,5/7)--(133/169,3/7)--(135/169,5/7)--(137/169,3/7)--(138/169,4/7)--(140/169,2/7)--(142/169,4/7)--(143/169,3/7)--(145/169,5/7)--(147/169,3/7)--(149/169,5/7)--(151/169,4/7)--(152/169,6/7)--(154/169,4/7)--(155/169,5/7)--(157/169,3/7)--(159/169,5/7)--(160/169,4/7)--(162/169,6/7)--(164/169,4/7)--(166/169,6/7)--(167/169,5/7)--(1,1);
	\draw[dashed] (70/169,5/7)--(70/169,0);
	\draw[dashed] (99/169,1/7)--(99/169,0);
	\end{tikzpicture}
	\caption{Simple $n$-crooked map
	s $\sigma_n$ for $n=6$ and $n=7$. Note that $\sigma_6$ and $\sigma_7$ are $3/6$ and $3/7$-crooked respectively.}\label{fig:sigma}
\end{figure} 

In what follows we define circle maps $\lambda_{n,k}$ that we will work with throughout the section. It will be sufficient for our purposes to define these maps for eventually every odd $n$. Due to the more particular construction we need in comparison to \cite{CO} more specific perturbations and thus we will require $k\geq 1$ to be even.

	Let $(\scr [n])^\infty_{n=1}\subset \N$ be the sequence defined in the following way: 
	$\scr [1]:=1$, $\scr [2]:=2$ and $\scr [n]:=2\scr [n-1]+\scr [n-2]$ for each $n\geq 3$.

In what follows let us denote by 
\begin{equation}\label{eq:eta}
\zeta:=\frac{\scr[n-1]}{2(\scr[n]+\scr[n-1])}.
\end{equation}

For each $n\in \N$, let \textit{simple $n$-crooked map}
$\sigma_n\colon [0,1]\to [0,1]$ be defined inductively as in \cite{LM} or \cite{CO} (see also Figure~\ref{fig:sigma} and Figure~\ref{fig:hatlambda}).\\
Let $\sigma_{-n}$ denote the reflection of the simple $n$-crooked map, that is 
$\sigma_{-n}(t):= 1 -\sigma_n (t)$ for each $t\in I$. Let $\sigma_{n}^{L}:=\sigma_{n}|_{[0,1/2]}$ where $\sigma_{n}|_{[0,1/2]}:[0,1/2]\to [0,\frac{n-1}{n}]$,  $\sigma_{n}^{R}:=\sigma_{n}|_{[1/2,1]}$ where $\sigma_{n}|_{[1/2,1]}:[1/2,1]\to [\frac{1}{n},1]$.  Similarly as above let $\sigma_{-n}^L$ and $\sigma_{-n}^R$ denote the reflections of $\sigma_{n}^L$ and $\sigma_{n}^R$ respectively. 

\begin{rem}\label{rem:crooked}
	By Proposition~3.5 in \cite{LM}, if $\eps>0$ and $n$ is sufficiently large to ensure $2/n<\eps$, the map $\sigma_n$ is $\eps$-crooked.
\end{rem}

For each odd integer $n\geq 7$ and each integer $k\geq 1$ define the map
$$
\hat{\lambda}_{n,k} \colon [0,n+k-1]\to \Big[0,\frac{2n+k-2}{n}\Big]
$$
by the formula
\begin{equation}
\label{eq:lambdahat}
\hat{\lambda}_{n,k}(t):=\begin{cases}
\frac{n-1}{n}\sigma^R_{-(n-1)}(\frac{t-i}{2\zeta}+\frac{1}{2})+\frac{i}{n},& \text{if } t\in[i,i+\zeta],\\
\sigma_n((t-i-\zeta)(\frac{1}{1-2\zeta}))+\frac{i}{n},& \text{if } t\in[i+\zeta,i+1-\zeta],\\
\frac{n-1}{n}\sigma_{-(n-1)}^L(\frac{t-i-1}{2\zeta}+\frac{1}{2})+\frac{i+1}{n},& \text{if } t\in[i+1-\zeta,i+1].\\
\end{cases}
\end{equation}
 for some $i=\{0,1,\ldots, n+k-2\}$. See Figure~\ref{fig:hatlambda} for the graph of $\hat{\lambda}_{n,k}|_{[i,i+1]}$ when $n=7$.
 
 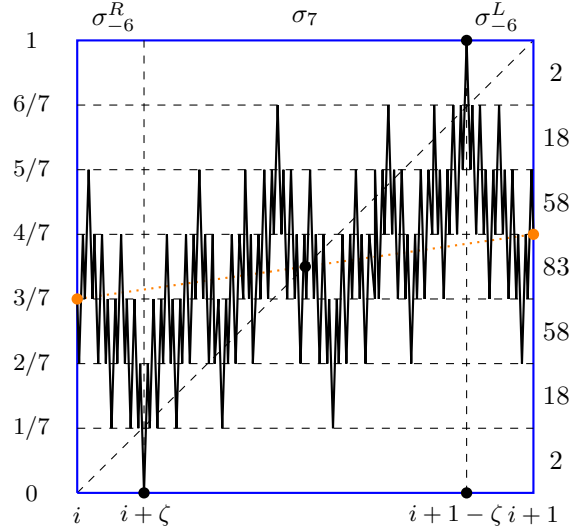
\begin{figure}[!ht]
 	\centering
 	\begin{tikzpicture}[scale=6]
 \draw[dashed] (0,0)--(1,1);
  \draw[dotted,orange,thick] (0,3/7)--(1,4/7);
 \node at (-0.1,0) {\small $0$};
 \node at (-0.1,1/7) {\small$1/7$};
 \node at (-0.1,2/7) {\small $2/7$};
 \node at (-0.1,3/7) {\small $3/7$};
 \node at (-0.1,4/7) {\small $4/7$};
 \node at (-0.1,5/7) {\small $5/7$};
 \node at (-0.1,6/7) {\small $6/7$};
 \node at (-0.1,1) {\small $1$};
 \draw[thick,blue] (0,0)--(0,1)--(1,1)--(1,0)--(0,0);
 \draw[dashed] (0,1/7)--(1,1/7);
 \draw[dashed] (0,2/7)--(1,2/7);
 \draw[dashed] (0,3/7)--(1,3/7);
 \draw[dashed] (0,4/7)--(1,4/7);
 \draw[dashed] (0,5/7)--(1,5/7);
 \draw[dashed] (0,6/7)--(1,6/7);
 \draw[thick](0,3/7)--(1/239,2/7)--(3/239,4/7)--(4/239,3/7)--(6/239,5/7)--(8/239,3/7)--(9/239,4/7)--(11/239,2/7)--(13/239,4/7)--(15/239,2/7)--(16/239,3/7)--(18/239,1/7)--(20/239,3/7)--(21/239,2/7)--(23/239,4/7)--(25/239,2/7)--(26/239,3/7)--(28/239,1/7)--(30/239,3/7)--(32/239,1/7)--(33/239,2/7)--(35/239,0)--(37/239,2/7)--(38/239,1/7)--(40/239,3/7)--(42/239,1/7)--(44/239,3/7)--(45/239,2/7)--(47/239,4/7)--(49/239,2/7)--(50/239,3/7)--(52/239,1/7)--(54/239,3/7)--(55/239,2/7)--(57/239,4/7)--(59/239,2/7)--(61/239,4/7)--(62/239,3/7)--(64/239,5/7)--(66/239,3/7)--(67/239,4/7)--(69/239,2/7)--(71/239,4/7)--(72/239,2/7)--(74/239,3/7)--(76/239,1/7)--(78/239,3/7)--(79/239,2/7)--(81/239,4/7)--(83/239,2/7)--(85/239,4/7)--(87/239,3/7)--(88/239,5/7)--(90/239,3/7)--(91/239,4/7)--(92/239,2/7)--(95/239,4/7)--(96/239,3/7)--(98/239,5/7)--(100/239,3/7)--(102/239,5/7)--(103/239,4/7)--(105/239,6/7)--(107/239,4/7)--(108/239,5/7)--(110/239,3/7)--(112/239,5/7)--(114/239,3/7)--(115/239,4/7)--(117/239,2/7)--(119/239,4/7)--(120/239,3/7)--(122/239,5/7)--(124/239,3/7)--(125/239,4/7)--(127/239,2/7)--(129/239,4/7)--(131/239,2/7)--(132/239,3/7)--(134/239,1/7)--(136/239,3/7)--(137/239,2/7)--(139/239,4/7)--(141/239,2/7)--(143/239,4/7)--(144/239,3/7)--(146/239,5/7)--(148/239,3/7)--(149/239,4/7)--(151/239,2/7)--(153/239,4/7)--(156/239,3/7)--(156/239,5/7)--(158/239,3/7)--(160/239,5/7)--(161/239,4/7)--(163/239,6/7)--(165/239,4/7)--(166/239,5/7)--(168/239,3/7)--(170/239,5/7)--(172/239,3/7)--(173/239,4/7)--(175/239,2/7)--(177/239,4/7)--(178/239,3/7)--(180/239,5/7)--(182/239,3/7)--(184/239,5/7)--(186/239,4/7)--(187/239,6/7)--(189/239,4/7)--(190/239,5/7)--(192/239,3/7)--(194/239,5/7)--(195/239,4/7)--(197/239,6/7)--(199/239,4/7)--(201/239,6/7)--(202/239,5/7)--(204/239,1)--(206/239,5/7)--(207/239,6/7)--(209/239,4/7)--(211/239,6/7)--(213/239,4/7)--(214/239,5/7)--(216/239,3/7)--(218/239,5/7)--(219/239,4/7)--(221/239,6/7)--(223/239,4/7)--(224/239,5/7)--(226/239,3/7)--(228/239,5/7)--(230/239,3/7)--(231/239,4/7)--(233/239,2/7)--(235/239,4/7)--(236/239,3/7)--(238/239,5/7)--(1,4/7);
 \draw[dashed] (35/239,0)--(35/239,1);
 \draw[dashed] (204/239,1)--(204/239,0);
 \node at (1/2,1.05) {$\sigma_7$};
 \node at (0.08,1.05) {$\sigma^R_{-6}$};
 \node at (0.92,1.05) {$\sigma^L_{-6}$};
  \node at (0,-0.05) {\small $i$};
 \node at (0.15,-0.05) {\small $i+\zeta$};
 \node at (0.83,-0.05) {\small $i+1-\zeta$};
  \node at (1,-0.05) {\small $i+1$};
 \node[circle,fill,orange, inner sep=1.5] at (0,3/7){};
 \node[circle,fill,orange, inner sep=1.5] at (1,4/7){};
 \node[circle,fill, inner sep=1.5] at (35/239,0){};
 \node[circle,fill, inner sep=1.5] at (204/239,1){};
  \node[circle,fill, inner sep=1.5] at (204/239,0){};
 \node[circle,fill, inner sep=1.5] at (0.5,0.5){};
  \node at (1.05,1/14) {$2$};
   \node at (1.05,3/14) {$18$};
    \node at (1.05,5/14) {$58$};
     \node at (1.05,7/14) {$83$};
      \node at (1.05,9/14) {$58$};
       \node at (1.05,11/14) {$18$};
        \node at (1.05,13/14) {$2$};
 	\end{tikzpicture}
 	\caption{This figure shows the building blocks of the function $\hat{\lambda}_{n,k}$ for $n=7$. The numbers on the right side of the graph represent the number of intervals of monotonicity of diameter $1/n$ in the respective horizontal strip. Counting the numbers of such intervals we will argue that maps $\lambda_{n,k}$ preserve Lebesgue measure. The dotted line represents the diagonal for the perturbation map $\lambda_{n,k}$.}\label{fig:hatlambda}
 \end{figure} 
 
 \begin{obs}\label{obs:lambdahat}
 For each odd integer $n\geq 7$ and each $k\in \mathbb{N}$ the map $\hat{\lambda}_{n,k}$ is 
 \begin{enumerate}
  \item a continuous and piecewise linear function with the uniform slope $\frac{\scr[n]+\scr[n-1]}{n}$.
  \item \label{obs:lambdahat2} an odd function around the point $\big(\frac{n+k-1}{2},\frac{2n+k-2}{2n}\big)$.
  \end{enumerate}
 \end{obs}

\begin{lem}\label{lem:hatcrooked}
	For every odd integer $n\geq 7$ and every integer $k\geq 1$ it holds that if $t,s\in [0,n+k-1]$ are such that $|\hat{\lambda}_{n,k}(t)-\hat{\lambda}_{n,k}(s)|<\frac{n-1}{n}$ then $\hat{\lambda}_{n,k}$ is $\frac{3}{n}$-crooked between $\hat{\lambda}_{n,k}(t)$ and $\hat{\lambda}_{n,k}(s)$. 
\end{lem}

In what follows, $[x]$ will denote the integer part of $x\in \R$, defined as the $\lfloor x \rfloor$ if $x$ is non-negative, and $\lceil x \rceil$ otherwise.

 \begin{defn}\label{def:lambda}
 	For every odd integer $n\geq 7$ and every $k\in\mathbb{N}$ define the map
 	$\lambda_{n,k} \colon \Ci\to \Ci$ through the lifting $\tilde\lambda_{n,k}:\mathbb{R}\to \mathbb{R}$ by 
 	\begin{equation}
 	\label{eq:lambda}
 	\tilde \lambda_{n,k}(t):=[t] +\frac{n}{n+k-1}\hat{\lambda}_{n,k}((t\hspace{-0.3cm}\pmod 1)(n+k-1))-\frac{n-1}{2(n+k-1)} \text{ for }  t\in 
    \mathbb{R}.
 	\end{equation}
 	Furthermore, we define $\lambda_{n,k,\alpha}:\mathbb{S}^1\to \mathbb{S}^1$ by $\lambda_{n,k,\alpha}:= r_{-\alpha}\circ \lambda_{n,k} \circ r_{\alpha}$ for any $\alpha\in \R$. 
 \end{defn}

See Figure~\ref{fig:lambda} for schematic picture of $\lambda_{n,k}$ for $n=7$ and note that the properly scaled building blocks of $\lambda_{n,k}$ are as on Figure~\ref{fig:hatlambda}.

\begin{obs}\label{obs:lambda1}
For each odd integer $n\geq 7$ and each integer $k\geq 1$ the map $\tilde \lambda_{n,k}$ is 
\begin{enumerate}
	\item \label{obs:lambda1(1)} a continuous piecewise linear function with the uniform slope being $\pm(\scr[n]+\scr[n-1])$; this claim holds equally well for $\tilde\lambda_{n,k,\alpha}$ for every $\alpha$,
	\item \label{obs:lambda1(2)} an odd function around the point $(1/2+i,1/2+i)$ for all $i\in \Z$,
	\item \label{obs:lambda1(3)} such that $\tilde \lambda_{n,k}\big(\frac{j}{n+k-1} + i\big)=\frac{j}{n+k-1}+i$ for all $j\in \{0,\ldots, n+k-1\}$ and all $i\in \Z$.
\end{enumerate}
\end{obs}

Now we will turn to the proof that the function $\lambda_{n,k}$ preserves Lebesgue measure for all odd $n\geq 7$ and all $k\geq 1$. We will do that based on the arguments on the representatives of circle maps.

\begin{defn}
	For every $j\in\{1,\ldots, n+k-1\}$ and some fixed integer $k\geq 1$ and an odd integer $n\geq 7$ denote by $I_j:=\frac{1}{n+k-1}[j-1,j)\subset [0,1)$ and let
	$$
	V_j:=I_j\times [0,1)
	$$
	denote the {\em $j$-th vertical strip} and let
	$$
	H_j:=[0,1)\times I_j
	$$
	denote the {\em $j$-th horizontal strip}.\\ 
	Let $p_1:[0,1)\times [0,1)\to [0,1)$ ($p_2:[0,1)\times [0,1)\to [0,1)$) be the natural projection onto the first (second) coordinate.
	Furthermore, define the maximal number of intervals of monotonicity of a representative $\Lambda_{n,k}:[0,1)\to [0,1)$ of a circle map $\lambda_{n,k}$ of the diameter exactly $\frac{1}{n+k-1}$ in $V_j$ by 
	$$
	\#(V_j) \text{ (resp. } \#(H_j)).
 	$$
Let us also define the subintervals  $\tilde{I}_j:=\frac{1}{n+k-1}[j-1,j)\subset \R$ that we will consider with respect to the liftings $\tilde \lambda_{n,k}$ of $\lambda_{n,k}$.
\end{defn}

\begin{obs}\label{obs:lambda2}
	 For each odd integer $n\geq 7$ and each integer $k\geq 1$ 
\begin{enumerate}
	\item \label{obs:lambda2(1)} the function $\tilde \lambda_{n,k}|_{\tilde I_j}$ is an odd function around the point $\Big(\frac{2j-1}{2(n+k-1)},\frac{2j-1}{2(n+k-1)}\Big)$ for any $j\in \Z$. 
	\item \label{obs:lambda2(3)}  $\#(V_j)=\scr[n]+\scr[n-1]$ for any $j\in  \{1,\ldots, n+k-1\}$ 
	(this follows from the definition of the representatives $\Lambda_{n,k}$ of $\lambda_{n,k}$, see the red part of the Figure~\ref{fig:lambda}).
	\item \label{obs:lambda2(4)} $\diam(\tilde \lambda_{n,k}({\tilde I_j}))=\frac{n}{n+k-1}$ for any $j\in \Z$.
\end{enumerate}
\end{obs}

\begin{prop}\label{prop:LebPres}
	For each odd integer $n\geq 7$ and each integer $k\geq 1$ and each $\alpha\in \mathbb{R}$ the map $\lambda_{n,k,\alpha}$ preserves Lebesgue measure on $\Ci$.
\end{prop}

\begin{proof}
First note that if $\lambda_{n,k}\in C_{\lambda}(\Ci)$, then $\lambda_{n,k,\alpha}\in C_{\lambda}(\Ci)$. Therefore, we just need to argue that $\lambda_{n,k}\in C_{\lambda}(\Ci)$.
Here we need to consider the representatives $\Lambda_{n,k}$ of $\lambda_{n,k}$ on the interval $[0,1)$ to show that $\lambda_{n,k}$ preserves Lebesgue measure on $\mathbb{S}^1$.
The proof is similar to the proof of Proposition 2.18. from \cite{CO} after noting that the flip function $\mathrm{Fl}$ from Definition 2.12. in \cite{CO} has the same effect on the number of maximal monotone injective branches of the same length of $\lambda_{n,k}$ as considering representatives of maps $\lambda_{n,k}$. Let us sketch the arguments for completeness.

From Observation~\ref{obs:lambda2} (\ref{obs:lambda2(3)}) it holds that $\#(V_j)=\scr[n]+\scr[n-1]$ for all $j\in \{1,\ldots, n+k-1\}$ and since by Observation~\ref{obs:lambda1} (\ref{obs:lambda1(1)}) $\lambda_{n,k}$ (and thus $\Lambda_{n,k}$ as well) has uniform slope (in the absolute value) we only need to show that $\#(H_j)=\scr[n]+\scr[n-1]$ since $\diam(p_1(V_{j}))=\diam(p_2(H_{j'}))$ for any $j,j'\in \{1,\ldots, n+k-1\}$.
Consider the {\em boxes} $B_{j,l}:=V_j\cap H_l$ and the number of injective branches of $\lambda_{n,k}$ of $B_{j,l}$ denoted by $\#(B_{j,l})$  for all $j,l\in\{1,\ldots, n+k-1\}$, see Figure~\ref{fig:lambda}.

First assume that $j\in \{\frac{n+1}{2},\ldots, k+\frac{n-1}{2}\}$. Note that $\#(V_j)=\#(B_{j,j-\frac{n-1}{2}})+\ldots+\#(B_{j,j})+\ldots +\#(B_{j,j+\frac{n-1}{2}})=\#(B_{j,j})+2\#(B_{j+1,j})+2\#(B_{j+2,j})+ \ldots +2\#(B_{j+\frac{n-1}{2},j})${, due to Observation~\ref{obs:lambda2} (\ref{obs:lambda2(1)}) and since $n$ is odd}. But note that (see the highlighted middle part of Figure~\ref{fig:lambda}) $\#(B_{j,m})=\#(B_{m,j})$ for all $m\in \{j-\frac{n-1}{2},\ldots ,j+\frac{n-1}{2}\}$ and thus $\#(H_j)=\#(B_{j-\frac{n-1}{2},j})+\ldots+\#(B_{j,j})+\ldots +\#(B_{j+\frac{n-1}{2},j})=\#(B_{j,j-\frac{n-1}{2}})+\ldots+\#(B_{j,j})+\ldots +\#(B_{j,j+\frac{n-1}{2}})=\#(V_j)$ which finishes this part of the proof.

Now assume that $j\in \{1,\ldots, \frac{n-1}{2}\} \cup \{k+\frac{n+1}{2},\ldots, n+k-1\}$. Due to the symmetricity of the map, by Observation~\ref{obs:lambda1} (\ref{obs:lambda1(2)}) it is enough to check the claim for $j\in \{1,\ldots, \frac{n+1}{2}\}$. By the definition of a representative and Observation~\ref{obs:lambda1} (\ref{obs:lambda1(2)}) (symmetricity of the map $\Lambda_{n,k}$) it holds that $\#(H_j)=\#(B_{1,j})+\#(B_{2,j})\ldots+\#(B_{j,j})+\ldots +\#(B_{j+\frac{n-1}{2},j})=\#(B_{j,j})+2\#(B_{j+1,j})+2\#(B_{j+2,j})+\ldots +2\#(B_{j+\frac{n-1}{2},j})=\#(V_j)$ (see the lower three horizontal strips of Figure~\ref{fig:lambda}){, which finishes this part of proof by Observation~\ref{obs:lambda2} (\ref{obs:lambda2(3)})}.\\
Thus for every non-degenerate interval $J\subset [0,1)$ we can use Equation~\eqref{e:3} and therefore due to Lemma~\ref{l:3} $\lambda_{n,k}\in C_{\lambda}(\mathbb{S}^1)$.
\end{proof}
\vspace{-0.2cm}
\begin{figure}[!ht]
	\centering
	\begin{tikzpicture}[scale=0.6]
	\draw[thick] (0,0)--(0,-3)--(1,-3)--(1,-2)--(2,-2)--(2,-1)--(3,-1)--(3,0)--(10,0)--(10,13)--(9,13)--(9,12)--(8,12)--(8,11)--(7,11)--(7,10)--(0,10)--(0,0);
	\draw [thick,blue] (4,8)--(5,8)--(5,1)--(4,1)--(4,8);
	\draw [thick,blue] (8,4)--(8,5)--(1,5)--(1,4)--(8,4);
	\draw[dashed] (0,1)--(10,1);
		\draw[dashed] (0,2)--(10,2);
		\draw[dashed] (0,3)--(10,3);
		\draw[dashed] (0,4)--(10,4);
		\draw[dashed] (0,5)--(10,5);
		\draw[dashed] (0,6)--(10,6);
		\draw[dashed] (0,7)--(10,7);
	\draw[dashed] (0,8)--(10,8);
	\draw[dashed] (0,9)--(10,9);
	\draw[dashed] (0,-2)--(1,-2);
		\draw[dashed] (0,-1)--(2,-1);
		\draw[dashed] (0,0)--(3,0);
		\draw[dashed] (7,10)--(10,10);
		\draw[dashed] (8,11)--(10,11);
		\draw[dashed] (9,12)--(10,12);
 \draw[dashed] (1,-3)--(1,10);
 \draw[dashed] (2,-2)--(2,10);
 \draw[dashed] (3,-1)--(3,10);
 \draw[dashed] (4,0)--(4,10);
 \draw[dashed] (5,0)--(5,10);
 \draw[dashed] (6,0)--(6,10);
 \draw[dashed] (7,0)--(7,11);
 \draw[dashed] (8,0)--(8,12);
 \draw[dashed] (9,0)--(9,13);
 \draw[dotted,thick,orange] (-2.5,-2.5)--(12.5,12.5);
 \node at (0.5,0.5) {$83$};
 \node at (1.5,1.5) {$83$};
 \node at (2.5,2.5) {$83$};
 \node at (3.5,3.5) {$83$};
  \node at (4.5,4.5) {$83$};
   \node at (5.5,5.5) {$83$};
    \node at (6.5,6.5) {$83$};
     \node at (7.5,7.5) {$83$};
      \node at (8.5,8.5) {$83$};
       \node at (9.5,9.5) {$83$};
 \node at (0.5,1.5) {$58$};
 \node at (1.5,2.5) {$58$};
 \node at (2.5,3.5) {$58$};
 \node at (3.5,4.5) {$58$};  
  \node at (4.5,5.5) {$58$}; 
   \node at (5.5,6.5) {$58$}; 
    \node at (6.5,7.5) {$58$}; 
     \node at (7.5,8.5) {$58$}; 
      \node at (8.5,9.5) {$58$}; 
  \node at (0.5,2.5) {$18$};         
   \node at (1.5,3.5) {$18$};
     \node at (2.5,4.5) {$18$};    
       \node at (3.5,5.5) {$18$};    
         \node at (4.5,6.5) {$18$};    
           \node at (5.5,7.5) {$18$};    
             \node at (6.5,8.5) {$18$};
              \node at (7.5,9.5) {$18$};
             
  \node at (0.5,3.5) {$2$};
   \node at (1.5,4.5) {$2$};
    \node at (2.5,5.5) {$2$};
     \node at (3.5,6.5) {$2$};
      \node at (4.5,7.5) {$2$};
       \node at (5.5,8.5) {$2$};
        \node at (6.5,9.5) {$2$};
  \node at (1.5,0.5) {$58$};
 \node at (2.5,1.5) {$58$};
 \node at (3.5,2.5) {$58$};
 \node at (4.5,3.5) {$58$};  
 \node at (5.5,4.5) {$58$}; 
 \node at (6.5,5.5) {$58$}; 
 \node at (7.5,6.5) {$58$}; 
 \node at (8.5,7.5) {$58$}; 
 \node at (9.5,8.5) {$58$}; 
 \node at (2.5,0.5) {$18$};         
 \node at (3.5,1.5) {$18$};
 \node at (4.5,2.5) {$18$};    
 \node at (5.5,3.5) {$18$};    
 \node at (6.5,4.5) {$18$};    
 \node at (7.5,5.5) {$18$};    
 \node at (8.5,6.5) {$18$};
 \node at (9.5,7.5) {$18$};
 \node at (3.5,0.5) {$2$};
 \node at (4.5,1.5) {$2$};
 \node at (5.5,2.5) {$2$};
 \node at (6.5,3.5) {$2$};
 \node at (7.5,4.5) {$2$};
 \node at (8.5,5.5) {$2$};
 \node at (9.5,6.5) {$2$};
              
 \node at (10.5,0.5) {$239$};
  \node at (10.5,1.5) {$239$};
   \node at (10.5,2.5) {$239$};
    \node at (10.5,3.5) {$239$};
     \node at (10.5,4.5) {$239$};
      \node at (10.5,5.5) {$239$};
       \node at (10.5,6.5) {$239$};
        \node at (10.5,7.5) {$239$};
         \node at (10.5,8.5) {$239$};
          \node at (10.5,9.5) {$239$};
            \node at (0.5,10.5) {$239$};
             \node at (1.5,10.5) {$239$};
              \node at (2.5,10.5) {$239$};
               \node at (3.5,10.5) {$239$};
                \node at (4.5,10.5) {$239$};
                 \node at (5.5,10.5) {$239$};
                  \node at (6.5,10.5) {$239$};
                   \node at (7.5,11.5) {$239$};
                    \node at (8.5,12.5) {$239$};
                     \node at (9.5,13.5) {$239$};
  \node at (0.5,9.5-10) {\red$58$};
 \node at (0.5,8.5-10) {\red$18$};
 \node at (0.5,7.5-10) {\red$2$};
 \node at (1.5,9.5-10) {\red$18$};
 \node at (1.5,8.5-10) {\red$2$};
 \node at (2.5,9.5-10) {\red$2$};
 
  \node at (9.5,10.5) {\red$58$};
 \node at (8.5,10.5) {\red$18$};
 \node at (7.5,10.5) {\red$2$};
 \node at (9.5,11.5) {\red$18$};
 \node at (8.5,11.5) {\red$2$};
 \node at (9.5,12.5) {\red$2$};
 
   \node at (0.5,9.5) {\red$58$};
 \node at (0.5,8.5) {\red$18$};
 \node at (0.5,7.5) {\red$2$};
 \node at (1.5,9.5) {\red$18$};
 \node at (1.5,8.5) {\red$2$};
 \node at (2.5,9.5) {\red$2$};
 
 \node at (9.5,0.5) {\red $58$};
 \node at (8.5,0.5) {\red$18$};
 \node at (7.5,0.5) {\red$2$};
 \node at (9.5,1.5) {\red$18$};
 \node at (8.5,1.5) {\red$2$};
 \node at (9.5,2.5) {\red$2$};
 
 \node[circle,fill,orange, inner sep=1.5] at (0,0){};
\node[circle,fill,orange, inner sep=1.5] at (1,1){};
\node[circle,fill,orange, inner sep=1.5] at (2,2){};
\node[circle,fill,orange, inner sep=1.5] at (3,3){};
\node[circle,fill,orange, inner sep=1.5] at (4,4){};
\node[circle,fill,orange, inner sep=1.5] at (5,5){};
\node[circle,fill,orange, inner sep=1.5] at (6,6){};
\node[circle,fill,orange, inner sep=1.5] at (7,7){};
\node[circle,fill,orange, inner sep=1.5] at (8,8){};
\node[circle,fill,orange, inner sep=1.5] at (9,9){};
\node[circle,fill,orange, inner sep=1.5] at (10,10){};
\node[circle,fill,orange, inner sep=1.5] at (11,11){};
\node[circle,fill,orange, inner sep=1.5] at (12,12){};
\node[circle,fill,orange, inner sep=1.5] at (-1,-1){};
\node[circle,fill,orange, inner sep=1.5] at (-2,-2){};
	\end{tikzpicture}
	\caption{The numbers in this picture represent the sums of the maximal numbers of injective branches of $\lambda_{n,k}$ of diameter $\frac{1}{\scr[n]+\scr[n-1]}$ where $n=7$ and $k=4$ 
	(note also that boxes divide the interval in $n+k-1=10$ pieces) in each box $B_{j,l}$ (the number in the boxes represents $\#(B_{j,l})$). The numbers outside of the boxes represent $\#(V_j)$ and $\#(H_l)$ respectively. The vertical blue rectangle denotes the position of a rescaled version of the map $\hat \lambda_{7,k}$ from Figure~\ref{fig:hatlambda} (its copies are placed all along the diagonal). Note that the sum of maximal injective branches is for every horizontal and vertical strip $239$ due to the fact that we work with a lifting of a circle map.}\label{fig:lambda}
\end{figure}
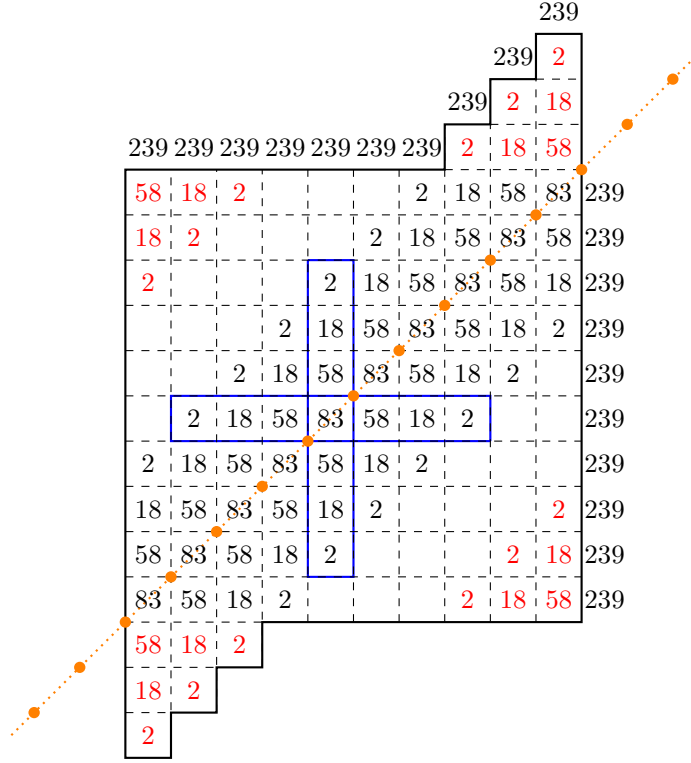

Now we will prove that $\lambda_{n,k}$ fits in the context of Proposition 5 of \cite{MT} (there such a perturbation map, which is however simpler, is denoted by $g$).

\begin{obs}\label{obs:gammaapart}
Let $x_j\in \tilde{I}_j$ be the minimal number such that
 $$\tilde\lambda_{n,k}(x_j):=\max\{\tilde\lambda_{n,k}(x); x\in \tilde I_j\}$$  for all  $j\in \Z$. 
Then
  $|x_{j-1}-x_{j}|=\frac{1}{n+k-1}$  and $|\tilde\lambda_{n,k}(x_{j-1})-\tilde\lambda_{n,k}(x_{j})|=\frac{1}{n+k-1}$ for all $j\in \Z$.
\end{obs}

\begin{obs}\label{obs:gammaapart2}
	If $n$ is odd and $k$ is even then $\lambda_{n,k}(t)+1/2=\lambda_{n,k}(t+1/2)$ for every $t\in \mathbb{S}^1$.
\end{obs}

The perturbations $\lambda_{n,k,\alpha}$ will be used later when we will prove that there exists a parameterized family of pseudo-circle attractors with continuously varying prime end rotation number.

We have the following result as a modification of Lemma 2.20 from \cite{CO}. Since both the statement and the proof need several modifications we give it in its entirety for the sake of completeness. Let us note that the following lemma is the part of the proof of main theorem from this section where we have the most changes comparing to the approach in \cite{CO}.

\begin{lem}\label{lem:MincUpdt}
	Let $\lambda_{n,k,\alpha}: \mathbb{S}^1\to \mathbb{S}^1$ be defined as in Definition~\ref{def:lambda}. 
	Set $\eps:=\frac{n-1}{n+k-1}$, $\gamma:=\frac{1}{n+k-1}$. Then the following statements hold  for every odd integer $n\geq 7$ and $k\geq 1$ and  $|\alpha|<\gamma/2$:
	\begin{enumerate}[(i)]
		\item\label{Lcra:1} $\rho(\lambda_{n,k,\alpha},\mathrm{id}_{\mathbb{S}^1})<\eps/2+\gamma+\alpha$,
		\item\label{Lcra:2} each map $\omega:[0,1]\to \mathbb{S}^1$ with $\lambda(\omega([0,1]))<\eps$  is $(\lambda_{n,k},3\gamma)$-crooked,
		\item\label{Lcra:3} for each subarc $A$ of $\mathbb{S}^1$ we have $\lambda(\lambda_{n,k,\alpha}(A)) \geq \lambda(A)$,
		and if, additionally, $\lambda(A)> \gamma$, then

		\begin{enumerate}[(a)]
			\item\label{Lcra:4} $\lambda(\lambda_{n,k,\alpha}(A)) > \eps/2$, 
			\item\label{Lcra:5}  $A \subset \lambda_{n,k,\alpha}(A)$ and
			\item\label{Lcra:6} $\lambda_{n,k,\alpha}(B) \subset B(\lambda_{n,k,\alpha}(A),r + \gamma)$ for each non-negative real number $r$ and each set $B \subset B(A,r)$.
		\end{enumerate}	
  
	\end{enumerate}
\end{lem}

\begin{proof}
	\eqref{Lcra:1} From the construction of the map $\lambda_{n,k}$ it follows that (see Figure~\ref{fig:hatlambda} and \ref{fig:lambda}): $$d(t-\lambda_{n,k}(t))<\frac{n+1}{2}\frac{1}{n+k-1}=\frac{n-1}{2(n+k-1)}+\frac{1}{n+k-1}=\eps/2+\gamma.$$ Therefore, $d(\lambda_{n,k}(r_{\alpha}(x)), r_{\alpha}(x))<\eps/2+\gamma$ and so $\rho(\lambda_{n,k,\alpha}(x),x)<\eps/2+\gamma+\alpha$ for every $x\in \mathbb{S}^1$, which proves \eqref{Lcra:1}. 
	
	\eqref{Lcra:2} Fix any  $\omega:[0,1]\to \mathbb{S}^1$ and assume that $\lambda(\omega([0,1]))<\eps$. Let us consider two cases. 
	Let $c:=\min\{\omega(0),\omega(1)\}\subset [0,1)$ and $d:=\max\{\omega(0),\omega(1)\}\subset [0,1)$.
	
  The first case is that
    $[c,d]\subset \omega([0,1])$.
  Since only values at endpoints are important for crookedness, without loss of generality we can view $\omega$ as a linear map
    $\omega:[0,1]\to [c,d]$ and identify $\lambda_{n,k,\alpha}\circ \omega$ with $\tilde \lambda_{n,k,\alpha}\circ \omega$.
Applying proper rescaling factor it follows from Lemma~\ref{lem:hatcrooked} that $\omega$ is $(\lambda_{n,k,\alpha},3\gamma)$-crooked. 

  The second case $[c,d]\not\subset \omega([0,1])$. But $\tilde \lambda_{n,k,\alpha}$ restricted to $[0,1]$ is actually
  $$\sigma_{-(n-1)}^R\sigma_n \sigma_{-(n-1)} \sigma_n \sigma_{-(n-1)}\sigma_n\ldots \sigma_{-(n-1)}\sigma_n\sigma_{-(n-1)}^L,$$
and therefore $\lambda_{n,k,\alpha}$ on circle represented by $[0,1)$ can be also identified with this map. But $\sigma_{-(n-1)}^L\sigma_{-(n-1)}^R=\sigma_{-(n-1)}$ and hence we may also see graph of $\lambda_{n,k,\alpha}$ as a cyclic repetition of $\sigma_n\sigma_{-(n-1)}$. Therefore, argument for the second case is the same as the argument for the first case, with a proper change of parametrization of $\mathbb{S}^1$. 
This proves that $\omega$ is $(\lambda_{n,k},3\gamma)$-crooked.

	\eqref{Lcra:3} 
	It is enough to prove the statements for $\tilde \lambda_{n,k}$, since they imply analogous statements for $\lambda_{n,k}$ because subarcs $A\subset \mathbb{S}^1$ are connected. In the following two paragraphs we will also additionally comment how switching from $\lambda_{n,k}$ to $\lambda_{n,k,\alpha}$ for $|\alpha|< \gamma/2$ affects the arguments we provide.\\
 Let $\tilde A:=[a,a']\subset \mathbb{R}$ denote a subinterval obtained through the lifting of a subarc $A\subset \mathbb{S}^1$. If $\tilde A$ is such that $\diam(\tilde A)< \frac{2}{(\scr[n]+\scr[n-1])(n+k-1)}$, then by Observation~\ref{obs:lambda1} (\ref{obs:lambda1(1)}) it holds that $\diam(\tilde\lambda_{n,k}(\tilde A))> (\scr[n]+\scr[n-1])\frac{\diam(\tilde A)}{2}>\diam(\tilde A)$. If a subinterval $\tilde A\subset \mathbb{R}$ is such that $\gamma>\diam(\tilde A)\geq \frac{2}{(\scr[n]+\scr[n-1])(n+k-1)}$, then it follows that $\diam(\tilde\lambda_{n,k}(\tilde A))\geq \gamma>\diam(\tilde A)$ because $\tilde A$ contains at least one full interval of monotonicity of the diameter of image being $\gamma$.
 Note that the estimates in this paragraph are independent of $\alpha$ and thus hold equally well for $\lambda_{n,k,\alpha}$ as $r_{\alpha}$ for any $\alpha$ is an isometry.
 
  Now assume that $\gamma\leq \diam (\tilde A)\leq 2\gamma$.  In  the following two paragraphs we will argue for the claims \eqref{Lcra:5} and \eqref{Lcra:3} simultaneously.  By Observation~\ref{obs:gammaapart} there are $x,y\in \tilde A$ such that $\tilde\lambda_{n,k}(y)=\max \tilde\lambda_{n,k} (\tilde I_j)$ and $\tilde\lambda_{n,k}(x)=\min \tilde\lambda_{n,k} (\tilde I_{j'})$ where $|j-j'|\leq 1$ and $\tilde A\cap \tilde I_j\neq \emptyset$,  $\tilde A\cap \tilde I_{j'}\neq \emptyset$. Note that $\tilde A$ is contained in at most three different adjacent intervals say $\tilde I_{i-1}\cup \tilde I_i\cup \tilde I_{i+1}$ and since $\diam(\tilde A)>\gamma$ it follows that $\tilde{\lambda}_{n,k}(\tilde A)$ covers at least the interval $\tilde I_{i-2}\cup \tilde I_{i-1}\cup \tilde I_i\cup \tilde I_{i+1}\cup \tilde I_{i+2}$. This means that $\tilde A\subset [\tilde\lambda_{n,k}(x),\tilde\lambda_{n,k}(y)]\subset \tilde\lambda_{n,k}(\tilde A)$, in particular $\diam(\tilde\lambda_{n,k}(\tilde A))\geq \diam(\tilde A)$.\\ 
Now, let us study $\lambda_{n,k,\alpha}(\tilde A)$ in the case when $\gamma\leq \diam (\tilde A)\leq 2\gamma$. Again say $\tilde A\subset \tilde I_{i-1}\cup \tilde I_i\cup \tilde I_{i+1}$ for some $i$.  
  Wlog assume $0\geq \alpha>-\gamma/2$.
  Denote $[x_i,y_i]:=\tilde I_i$.
  
  First assume that $d(a,x_i)<\gamma/2$. Note that we are rotating to the left and $\diam \tilde I_{i-1}=\gamma$. 
  Then, since $\diam(r_{\alpha}(\tilde A))=\diam(\tilde A)\leq 2\gamma$ it follows that $r_{\alpha}(\tilde A)\subset \tilde I_{i-1}\cup \tilde I_i\cup \tilde I_{i+1}$. Following the arguments in the previous paragraph $\tilde{\lambda}_{n,k}(r_{\alpha}(\tilde A))$  covers at least the interval $\tilde I_{i-2}\cup \tilde I_{i-1}\cup \tilde I_i\cup \tilde I_{i+1}\cup \tilde I_{i+2}$. 
   But since $|\alpha|<\gamma/2<\gamma=\diam(I_{i-1})$ and noting that $r_{-\alpha}$ applied from the left shifts the values $\tilde \lambda_{n,k}(r_{\alpha(\tilde A)})$ for at most $\gamma/2$ it follows again from the arguments in the previous paragraph that $\tilde{\lambda}_{n,k,\alpha}(\tilde{A})=r_{-\alpha}(\tilde{\lambda}_{n,k}(r_{\alpha}(\tilde A)))$  covers at least the interval $\tilde I_{i-1}\cup \tilde I_{i}\cup  \tilde I_{i+1}$ and thus $\tilde A\subset [\tilde\lambda_{n,k,\alpha}(x),\tilde\lambda_{n,k,\alpha}(y)]\subset \tilde\lambda_{n,k,\alpha}(\tilde A)$.
   
    Now assume that $d(a,x_i)\geq \gamma/2$ and $\tilde A\cap \tilde I_{i-1}\neq \emptyset$. 
  Then we have that $r_{\alpha}(\tilde A)\subset \tilde I_{i-2}\cup \tilde I_{i-1}\cup \tilde I_{i}$.
Following the arguments as above $\tilde{\lambda}_{n,k}(r_{\alpha}(\tilde A))$  covers at least the interval $\tilde I_{i-3}\cup \tilde I_{i-2}\cup \tilde I_{i-1}\cup \tilde I_{i}\cup \tilde I_{i+1}$.
But now, with analogous arguments as in the second paragraph of the proof of \eqref{Lcra:3}, we can only claim that
   $\tilde{\lambda}_{n,k,\alpha}(\tilde{A})=r_{-\alpha}(\tilde{\lambda}_{n,k}(r_{\alpha}(\tilde A)))$  covers at least the interval $\overline{B(\tilde I_{i-1}\cup \tilde I_{i},\gamma/2)}$. However, since  $d(a,x_i)\geq \gamma/2$ it follows $\tilde A\subset \overline{B(\tilde I_{i-1}\cup \tilde I_{i},\gamma/2)}$.
   Now assume that $d(a,x_i)\geq \gamma/2$ and $\tilde A\cap \tilde I_{i-1}=\emptyset$. Thus 
   we have that $r_{\alpha}(\tilde A)\subset \tilde I_{i-1}\cup \tilde I_{i}\cup \tilde I_{i+1}$. But now we can use analogous arguments as in the second paragraph of the proof of \eqref{Lcra:3}.

   That $\diam(\tilde\lambda_{n,k,\alpha}(\tilde A))\geq \diam(\tilde A)$ now follows in both of the cases from the corresponding claim for $\lambda_{n,k}$ and the fact that $r_{\alpha}$ is an isometry.

	When $\diam(\tilde A)>2\gamma$, then there are $j\leq j'$ such that $\tilde I_j\cup \ldots \cup \tilde I_{j'}\subset \tilde A$ and $\diam(\tilde I_j\cup \ldots \cup \tilde I_{j'})$ is maximal possible for the interval $\tilde I_j\cup \ldots \cup \tilde I_{j'}$ under inclusion (meaning one cannot take smaller $j$ or larger $j'$). But then clearly $\tilde A\subset [\min\tilde\lambda_{n,k}(\tilde I_j), \max \tilde\lambda_{n,k}(\tilde I_{j'})]$ (see Figure~\ref{fig:lambda}) which completes the proof for $\tilde \lambda_{n,k}$ and thus for $\lambda_{n,k}$.\\
 By the definition we have
 $B(\tilde I_j\cup \ldots \cup \tilde I_{j'},3\gamma)\subset \tilde \lambda_{n,k}(\tilde I_j\cup \ldots \cup \tilde I_{j'})$. 
Furthermore $\lambda_{n,k}(\tilde I_j\cup \ldots \cup \tilde I_{j'})\subset B(\tilde{\lambda}_{n,k}(r_{\alpha}(\tilde A)),\gamma)$,
hence $B(\tilde I_j\cup \ldots \cup \tilde I_{j'},2\gamma)\subset \tilde{\lambda}_{n,k}(r_{\alpha}(\tilde A))$.
Finally, we have 
$$
\tilde A \subset B(r_{-\alpha}(\tilde I_j\cup \ldots \cup \tilde I_{j'}),2\gamma)\subset r_{-\alpha}(\tilde{\lambda}_{n,k}(r_{\alpha}(\tilde A)))=
\tilde{\lambda}_{n,k,\alpha}(\tilde A).
$$ 
  
Moreover, we again get that $\diam(\tilde\lambda_{n,k,\alpha}(\tilde A))\geq \diam(\tilde A)$ from the analogous claim for $\lambda_{n,k}$ and the fact that $r_{\alpha}$ is an isometry.
	
	\eqref{Lcra:4} This part follows from the arguments in the last paragraph if one notes that $\eps/2=\frac{n-1}{2(n+k-1)}$ and the fact that $r_{\alpha}$ is an isometry. Namely, $|\tilde{\lambda}_{n,k,\alpha}(x)-\tilde{\lambda}_{n,k,\alpha}(y)|\geq\frac{n+1}{n+k-1}>\eps/2$ where $x$ and $y$ are as in the previous paragraph. 
	
\eqref{Lcra:6} First note that if the claim holds for $\lambda_{n,k}$ it also holds for $\lambda_{n,k,\alpha}$ for any $\alpha\in \R$. 
Namely, in intervals $\tilde I_j$ the first maximal value of $\tilde\lambda_{n,k}$ lies $\gamma$ apart for all $j\in \Z$  and it is exactly $\gamma$ greater from the maximal value of $\tilde\lambda_{n,k}(\tilde I_{j-1})$.
Furthermore, if we denote $\beta:=\frac{\gamma}{\scr[n]+\scr[n-1]}$; that is, $\beta$ is the length of the smallest interval of monotonicity, then $\tilde\lambda_{n,k}(x_i)=\max \tilde\lambda_{n,k}([x_i,x_{i+1}-\beta])$
for each $i\in \Z$, where $x_i$ are as in  Observation~\ref{obs:gammaapart}. Thus for the part when $\gamma \leq \diam(A)<2\gamma$ the conclusion follows from Observation~\ref{obs:gammaapart}.\\
Let $A\subset \mathbb{R}$ be an interval such that $\diam(A)\geq 2\gamma$, denote $A:=[a,b]$. Then there is maximal $j$ such that $\tilde I_j\subset A$.
Choose an $r\in \R$ and fix any $x\in [b,b+r]$. Then, by the repetitive structure of building blocks of the graph $\tilde\lambda_{n,k}$ (see Figure~\ref{fig:lambda}), if we take maximal non-negative integer $M$ such that $M\gamma <r+\gamma$,
$y:=x-M\gamma \in [a,b]$ and $\tilde\lambda_{n,k}(x)\leq \tilde\lambda_{n,k}(y)+M\gamma$. But this shows that $\tilde\lambda_{n,k}([a,b+r])\subset B(\tilde\lambda_{n,k}([a,b]),r+\gamma)$. Thus we have the required statement for the liftings and through the projection to $\mathbb{S}^1$ we have it for $\lambda_{n,k}$ as well.
\end{proof}

{

\begin{defn}
	A piecewise linear circle map $f\in C(\mathbb{S}^1)$ is called \emph{admissible}, if $|\tilde{f}'(t)|\geq 4$ for every $t\in \mathbb{S}^1$ for which $\tilde{f}'(t)$ exists and $f$ is leo.
\end{defn}

A more general version of the following lemma is proven in its entirety in Lemma~\ref{lem:LemMTadjusted} below.

\begin{lem}\label{lem:LemMT}
Let $f:\Ci\to \Ci$ be an admissible map. Let $\eta$ and $\delta$ be two positive real numbers. Then there is an admissible map $F:\Ci\to \Ci$ and a positive integer $n$ such that any $\omega:[0,1]\to \mathbb{S}^1$ is $(F^n,\delta)$-crooked and $\rho(F,f)<\eta$. Moreover, if $f\in C_{\lambda}(\Ci)$, such $F$ can also be chosen to be in $C_{\lambda}(\Ci)$.
\end{lem}

The following is a restatement of Lemma 3.2 from \cite{KTT} in the case $\mathbb{S}^1$.

\begin{lem}[{\cite[Lemma~ 3.2]{KTT}}]\label{lem:adm}
	For every $\eps>0$ and every leo map $f \in C(\Ci)$ there exists $F\in C(\Ci)$ such that $F$ is admissible and $\rho(F,f) < \eps$.
\end{lem}

We will need a small adjustment of Lemma~\ref{lem:adm} to the context of Lebesgue measure-preserving circle maps, which can be obtained with the help of the following useful result. 

   A proof of the following lemma follows the lines of the proof of Proposition 12 from \cite{BT} done for a representative of a circle map together with noting that Lebesgue measure-preserving circle maps are surjective.
   
By $\mathrm{PLM}_{\lambda\mathrm{(leo)}}$ we denote the set of circle piecewise linear leo maps that additionally satisfy the {\em Markov property}, which means there is a partition $0=a_0<a_1<\ldots<a_n=1$ such that for each $i$ the map $f_{[a_i,a_{i+1}]}$ is monotone and there are $s<t$ such that
$f([a_i,a_{i+1}])=[a_s,a_t]$.

\begin{lem}\label{lem:Markov}
	The set $\mathrm{PLM}_{\lambda\mathrm{(leo)}}$ is dense in $C_{\lambda}(\Ci)$.
\end{lem}

In what follows we will also need perturbations of maps that preserve Lebesgue measure, similarly as in \cite{BT}.

For an arc $A \subset \mathbb{S}^1$, let a {\em partition} $\{A_i \subset \mathbb{S}^1: 1 \le i \le m \}$ of $A$, where $m$ is odd, be a finite collection of arcs satisfying $\cup^{m}_{i=1} A_i = A$ and $\Int(A_i) \cap \Int(A_j) = \emptyset $ when $i \ne j$. Fix $f \in C_{\lambda}(\mathbb{S}^1)$
 an arc $A \in \mathbb{S}^1$ and a partition of $A$. Let $A^c$ denote the complement of $A$ in $\mathbb{S}^1$.
 A map $h = h_{A,\{A_i\}}$ is {\em an $m$-fold window perturbation of $f$ with respect to $A$ and the partition $\{A_i\}$} if
 \begin{itemize}
     \item $h|_{A^c} = f|_{A^c}$ 
     \item for every $1 \le i \le m$ the map $h|_{A_i}$ is a linearly scaled copy of $f|_{A}$ with the orientation reversed for every second index $i$, with $h|_{A_1}$ having the same orientation as $f|_{A}$.
 \end{itemize}
 The essence of this definition is illustrated by Figure~\ref{fig:perturbations}.\\
We call an $m$-fold window perturbation of $f$ {\em regular} if all of the $A_i$'s have the same length.

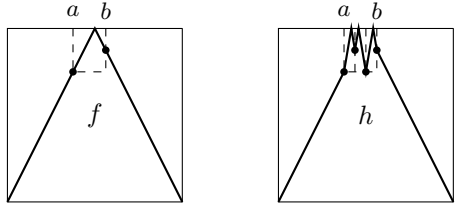
\begin{figure}[!ht]
	\centering
	\begin{tikzpicture}[scale=2.3, rotate=180]
	\draw (0,0)--(0,1)--(1,1)--(1,0)--(0,0);
	\draw[thick] (0,1)--(1/2,0)--(1,1);
	\node at (1/2,1/2) {$f$};
	\node at (7/16,-0.1) {\small $b$};
	\node at (5/8,-0.1) {\small $a$};
	\draw[dashed] (7/16,0)--(7/16,1/4)--(5/8,1/4)--(5/8,0);
	\node[circle,fill, inner sep=1] at (7/16,1/8){};
	\node[circle,fill, inner sep=1] at (5/8,1/4){};
	\end{tikzpicture}
	\hspace{1cm}
	\begin{tikzpicture}[scale=2.3, rotate=180]
	\draw (0,0)--(0,1)--(1,1)--(1,0)--(0,0);
	\draw[thick] (0,1)--(7/16,1/8)--(14/32+1/48,0)--(1/2,1/4)--(1/2+1/24,0)--(9/16,1/8)--(9/16+1/48,0)--(5/8,1/4)--(1,1);
	\draw[dashed] (1/2,0)--(1/2,1/4);
	\draw[dashed] (9/16,1/4)--(9/16,0);
	\node at (1/2,1/2) {$h$};
	\node at (7/16,-0.1) {\small $b$};
	\node at (5/8,-0.1) {\small $a$};
	\draw[dashed] (7/16,0)--(7/16,1/4)--(5/8,1/4)--(5/8,0);
	\node[circle,fill, inner sep=1] at (1/2,1/4){};
	\node[circle,fill, inner sep=1] at (9/16,1/8){};
	\node[circle,fill, inner sep=1] at (7/16,1/8){};
	\node[circle,fill, inner sep=1] at (5/8,1/4){};
	\end{tikzpicture}
	\caption{For  $f\in C_{\lambda}(\mathbb{S}^1)$ and $a,b\in I$ shown on the left picture, we show on the right picture
 the regular $3$-fold window perturbation of $f$.}\label{fig:perturbations}
\end{figure}
For more detail on the window perturbations we refer the reader to \cite{BT}.

\begin{lem}\label{lem:LeoApproxByAdm}
	For every $\eps>0$ and every leo map $f \in C_\lambda(\Ci)$ there exists $F\in C_\lambda(\Ci)$ such that $F$ is admissible and $\rho(F,f) < \eps$.
\end{lem}
\begin{proof}
By Lemma~\ref{lem:Markov} piecewise linear and Markov leo maps are dense in $C_{\lambda}(\Ci)$, so let us start with such a map $g$ with $\rho(g,f) < \eps/2$.
Let us choose a Markov partition $\mathcal{M}:=\{a_0,a_1,\ldots, a_n, a_{n+1}\}\subset \mathbb{S}^1$ where  $0=\bar a_{n+1}=\bar a_0<\bar a_1<\ldots<\bar a_n$ for points $\{\bar a_0,\ldots, \bar a_n, \bar a_{n+1}\}\subset [0,1)$ associated to points from $\mathcal{M}$ and a representative $G:[0,1)\to [0,1)$ associated to $g$. Periodic points are dense for a leo circle map, so including points from periodic orbits as points in the partition, we may also require that $|a_{i+1}-a_i|<\delta$ for any fixed $\delta>0$.
In particular, we have that $g$ is monotone on each arc $A\subset \mathbb{S}^1$ with endpoints $a_i$ and $a_{i+1}$, $\diam(g(A))<\eps/2$ and $g(A)=A'$ where $A'$ has endpoints $a_k$ and $a_{k'}$ for some indices $k<k'$.
Now, following the construction in Lemma~5 of \cite{BT} we construct a new map $F$ by replacing each $g|_{A}$ by its regular $m$-fold window perturbation, with odd
and sufficiently large $m$. This way $F$ is admissible with $\rho(F,g)\leq \diam(g(A))<\eps/2$.
Window perturbations are invariant for $C_\lambda(\Ci)$, hence $F\in C_\lambda(\Ci)$.
Clearly also $g(A)=F(A)=B$. Therefore, for each $i$ there is $n\in \N$ such that $F^n(A)=\mathbb{S}^1$.
But then, if we fix any open set $U\subset \mathbb{S}^1$, then since slope on intervals of monotonicity of $F$ is at least $4$, there is $M\in \mathbb{N}$ such that $F^M(U)$ contains three consecutive intervals of monotonicity, and therefore 
$F^{M+1}(U)\supset A$.
\end{proof}

The following lemma is a restatement of Lemma 2.4 from \cite{KTT} in the special case of $\mathbb{S}^1$.
\begin{lem}[{\cite[Proposition~2]{MT}}]\label{lem:crooked}
	Let $f,F\in C(\Ci)$ be two maps so that
	$\rho(f,F)<\varepsilon$. If $f$ is $\delta$-crooked, then $F$ is
	$(\delta+2\varepsilon)$-crooked.
\end{lem}

Now we are ready to prove the main theorem of this section.

\begin{proof}[Proof of Theorem \ref{thm:UniLimPresLeb}]
For any $k\geq 1$ let the set $A_k\subset C_\lambda(\Ci)$ be contained in the set of maps $\mathcal{A}$ such that for every $f\in \mathcal{A}$ it holds $f^n$ is $(1/k-\delta)$-crooked for some $n$ and some sufficiently small $\delta>0$. 
First observe that $A_k$ is dense in $C_\lambda(\Ci)$. Namely,  by Lemma~\ref{lem:Markov} it holds that piecewise linear leo Markov maps are dense in $C_{\lambda}(\Ci)$.
If we start with such a map $g$ then first applying Lemma~\ref{lem:LeoApproxByAdm} and next Lemma~\ref{lem:LemMT} we modify $g$ to a map $f\in A_k$
using an arbitrarily small perturbation. But if $f\in A_k$ and $n,\delta$ are constants from the definition of $A_k$, then by Lemma~\ref{lem:crooked} we have $B(f,\delta/4)\subset A_k$.
This shows that $A_k$ contains an open dense set. But then the set
$$
\mathcal{S}:=\bigcap_{k=1}^\infty A_k
$$
is a dense $G_\delta$ and clearly each element $\mathcal{S}\subset \mathcal{T}$ so $\mathcal{T}$ satisfies the conclusion of the theorem. 
\end{proof}

\section{Family of pseudo-circles with continuously varying prime ends rotation numbers}\label{sec:family}

 \begin{figure}[!ht]
 	\centering
 	\begin{tikzpicture}[scale=5]
 \draw(0,1)--(0,0)--(1,0);
 \draw[dashed](1,0)--(1,1)--(0,1);
 \draw[thick](0,1/2)--(3/130,2/10)--(4/130,3/10)--(5/130,2/10)--(6/130,3/10)--(7/130,2/10)--(8/130,3/10)--(9/130,2/10)--(1/10,6/10)--(1/10+3/130,1/10+2/10)--(1/10+4/130,1/10+3/10)--(1/10+5/130,1/10+2/10)--(1/10+6/130,1/10+3/10)--(1/10+7/130,1/10+2/10)--(1/10+8/130,1/10+3/10)--(1/10+9/130,1/10+2/10)--(1/10+1/10,1/10+6/10)--(2/10+3/130,2/10+2/10)--(2/10+4/130,2/10+3/10)--(2/10+5/130,2/10+2/10)--(2/10+6/130,2/10+3/10)--(2/10+7/130,2/10+2/10)--(2/10+8/130,2/10+3/10)--(2/10+9/130,2/10+2/10)--(2/10+1/10,2/10+6/10)   --(3/10+3/130,3/10+2/10)--(3/10+4/130,3/10+3/10)--(3/10+5/130,3/10+2/10)--(3/10+6/130,3/10+3/10)--(3/10+7/130,3/10+2/10)--(3/10+8/130,3/10+3/10)--(3/10+9/130,3/10+2/10)--(3/10+1/10,3/10+6/10)--(4/10+3/130,4/10+2/10)--(4/10+4/130,4/10+3/10)--(4/10+5/130,4/10+2/10)--(4/10+6/130,4/10+3/10)--(4/10+7/130,4/10+2/10)--(4/10+8/130,4/10+3/10)--(4/10+9/130,4/10+2/10)--(4/10+1/10,4/10+6/10)--(5/10+3/130,5/10+2/10)--(5/10+4/130,5/10+3/10)--(5/10+5/130,5/10+2/10)--(5/10+6/130,5/10+3/10)--(5/10+7/130,5/10+2/10)--(5/10+8/130,5/10+3/10)--(5/10+9/130,5/10+2/10)--(5/10+1/10-1/130,1);
 \draw[thick] (8/10-3/130,0)--(8/10,1/2-2/10)--(8/10+3/130,2/10-2/10)--(8/10+4/130,3/10-2/10)--(8/10+5/130,2/10-2/10)--(8/10+6/130,3/10-2/10)--(8/10+7/130,2/10-2/10)--(8/10+8/130,3/10-2/10)--(8/10+9/130,2/10-2/10)--(9/10,1/2-1/10)--(9/10+3/130,2/10-1/10)--(9/10+4/130,3/10-1/10)--(9/10+5/130,2/10-1/10)--(9/10+6/130,3/10-1/10)--(9/10+7/130,2/10-1/10)--(9/10+8/130,3/10-1/10)--(9/10+9/130,2/10-1/10)--(1,6/10-1/10);
 \draw[thick] (6/10+1/130,1)--(6/10+3/130,6/10+2/10)--(6/10+4/130,6/10+3/10)--(6/10+5/130,6/10+2/10)--(6/10+6/130,6/10+3/10)--(6/10+7/130,6/10+2/10)--(6/10+8/130,6/10+3/10)--(6/10+9/130,6/10+2/10)--(6/10+1/10-2/130,1);
  \draw[thick] (7/10+2/130,1)--(7/10+3/130,7/10+2/10)--(7/10+4/130,7/10+3/10)--(7/10+5/130,7/10+2/10)--(7/10+6/130,7/10+3/10)--(7/10+7/130,7/10+2/10)--(7/10+8/130,7/10+3/10)--(7/10+9/130,7/10+2/10)--(7/10+1/10-3/130,1);
  \draw[thick] (6/10-1/130,0)--(6/10,1/10)--(6/10+1/130,0);
   \draw[thick] (7/10-2/130,0)--(7/10,2/10)--(7/10+2/130,0);
   \draw[dotted,thick] (1/2,0)--(1,1/2);
 \draw[dashed] (0,1/10)--(1,1/10);
  \draw[dashed] (0,2/10)--(1,2/10);
   \draw[dashed] (0,3/10)--(1,3/10);
    \draw[dashed] (0,4/10)--(1,4/10);
     \draw[dashed] (0,5/10)--(1,5/10);
      \draw[dashed] (0,6/10)--(1,6/10);
       \draw[dashed] (0,7/10)--(1,7/10);
        \draw[dashed] (0,8/10)--(1,8/10);
         \draw[dashed] (0,9/10)--(1,9/10);
         
          \draw[dashed] (1/10,0)--(1/10,1);
          \draw[dashed] (2/10,0)--(2/10,1);
          \draw[dashed] (3/10,0)--(3/10,1);
          \draw[dashed] (4/10,0)--(4/10,1);
          \draw[dashed] (5/10,0)--(5/10,1);
          \draw[dashed] (6/10,0)--(6/10,1);
          \draw[dashed] (7/10,0)--(7/10,1);
          \draw[dashed] (8/10,0)--(8/10,1);
          \draw[dashed] (9/10,0)--(9/10,1);
          \draw[dotted,thick] (0,0)--(1,1);
          \draw[dotted,thick] (0,1/2)--(1/2,1);
          \draw[thick] (0,0)--(0.1,0);
          \node at (0.05,-0.05) {$\tilde\gamma$};
 \end{tikzpicture} 
  
\caption{A representative of the circle map $f$ defined in the beginning of Section~\ref{sec:family}. This map is the starting point of the construction of parameterized family of pseudo-circle attractors.} \label{fig:mapf}
 \end{figure}
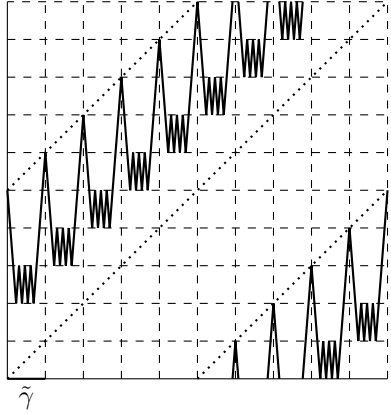
 First we show an auxiliary lemma that will be used later in this section.
 
 \begin{lem}
Fix an $\eps>0$.
Each $\omega:[0,1]\to\mathbb{S}^1$ with $\diam(\omega([0,1]))<\eps$ is $(f,\gamma)$-crooked, where $f\colon \mathbb{S}^1 \to  \mathbb{S}^1$ if and only if each $\omega$ is $(r_\alpha\circ f\circ r_{\alpha'},\gamma)$-crooked for every $\alpha,\alpha'\in \R$.
\end{lem}
\begin{proof}
Note that $\diam(\omega([0,1]))=\diam(r_{\alpha'}(\omega([0,1])))$
so provided $\omega$ is $(f,\gamma)$-crooked it follows $\omega$ is also $(f\circ  r_{\alpha'},\gamma)$-crooked and so it is $(r_\alpha\circ f\circ r_{\alpha'},\gamma)$-crooked since $r_\alpha$ is an isometry. The converse is obtained by putting $r_\alpha\circ f \circ r_{\alpha'}$ in place of $f$ and using inverse rotations $r_{-\alpha}$, $r_{-\alpha'}$.
\end{proof}

Let $f:\Ci\to \Ci$ be a circle map of degree one as on Figure~\ref{fig:mapf} (we parametrize the circle with $[0,1)$) and let $\{f_{\beta}:=r_{\beta}\circ f\}_{\beta\in\R}$. Let $\tilde{\gamma}:=1/5$; it is the diameter of "one block of the graph of function $f$" and let $\tilde{s}:=13$ denote the absolute value of the constant slope of $f$ (wherever defined); it is the same as the slope of $f_{\beta}$. 

\begin{lem}\label{lem:fbetaleo}
Map $f_{\beta}$ is leo for every $\beta\in\R$. Furthermore, $f_{\beta}$ is admissible for every $\beta\in \R$.
\end{lem}
\begin{proof}
Say $J\subset \Ci$ is an open arc such that $\diam(J)<\tilde{\gamma}$. Since $\tilde{s}=13$ there exists $N\in \mathbb{N}$ such that $\diam(f^{N}_{\beta}(J))$ covers three critical points of $f_{\beta}$. But then $\diam(f^{N+1}_{\beta}(J))\geq \tilde{\gamma}$ and $\diam(f^{N+2}_{\beta}(J))\geq 3\tilde{\gamma}$. 
Continuing inductively, for every $i\geq 2$ $\diam(f^{N+i}_{\beta}(J))\geq \min\{1,(i+2)\tilde{\gamma}\}$. Since $\tilde \gamma=1/10$ it therefore follows that $\diam(f^{N+8}_{\beta}(J))=1$ in the worst case scenario (we get $8$ since $f_{\beta}^{N+1}(J)$ covers at least $3\tilde{\gamma}$ (see Figure~\ref{fig:mapf}) and thus we inductively need at least $7$ more iterates to cover whole $\mathbb{S}^1$). That finishes the proof that $f_{\beta}$ is leo for all $\beta\in\R$. The furthermore claim follows from the additional fact that $\tilde s>4$ for all $\beta$.
\end{proof}

The proof of the following lemma uses a strategy of the proof of Lemma from \cite{MT} but in a more general setting from several aspects. First, instead of $f$ we have a composition of several maps $F_{\beta}$ and instead of a sole perturbation $g$ we also have a composition of several perturbations $\lambda_{n_1,k_1,\alpha_1}\circ \ldots \circ \lambda_{n_m,k_m,\alpha_m}$. Furthermore, we are proving the following lemma for $\mathbb{S}^1$ instead of the unit interval $[0,1]$.
An alternative approach to prove the following lemma would be to take a general strategy as in \cite{KTT}, however their construction requires different constants depending on the maps in the construction as we need later in our setting and thus our perturbations $\lambda_{n,k}$ do not fit in their requirements. 

\begin{lem}\label{lem:LemMTadjusted}
Fix sequences of odd positive integers $n_1,\ldots, n_m\geq 1$, positive integers $k_1,\ldots, k_m\geq 1$ and $\alpha_1,\alpha_2,\ldots ,\alpha_m\in (-1/2,1/2)$ such that $|\alpha_i|<\frac{1}{2(k_{i}+n_{i}-1)}$ for all $i\in \{1,\ldots,m\}$ and let $\eta$ and $\delta$ be two positive real numbers fixed for the whole parameterized family  of Lebesgue measure-preserving maps $\{F_{\beta}\}_{\beta\in\R}$ where $F_{\beta}:=f_{\beta}\circ \lambda_{n_1,k_1,\alpha_1}\circ\ldots \circ \lambda_{n_m,k_m,\alpha_m}$. 
Then there are even $k_{m+1}\geq 1$, odd $n_{m+1}\geq 1$ and $N\geq 1$ such that for every $|\alpha|<\frac{1}{2(k_{m+1}+n_{m+1}-1)}$ and every $\beta\in \R$
if we denote $G_{\beta,\alpha, m+1}:=F_{\beta}\circ \lambda_{n_{m+1},k_{m+1},\alpha}$ then every $\omega:[0,1]\to\Ci$ is $(G^N_{\beta,\alpha,m+1},\delta)$-crooked and $\rho(G_{\beta,\alpha,m+1},F_{\beta})<\eta$. Furthermore, each $G_{\beta,\alpha,m+1}\in C_{\lambda}(\Ci)$ is admissible.
\end{lem}
\begin{proof}

	As we are accustomed already, let $f:\mathbb{S}^1\to\mathbb{S}^1$ be an admissible map and $\tilde f:\mathbb{R}\to \mathbb{R}$ its lifting. 
	First observe that $\lambda_{n_1,k_1,\alpha_1}\circ\ldots \circ \lambda_{n_m,k_m,\alpha_m}$ remain unchanged in the formula for $F_{\beta}$ for each $\beta$. 
	
	As $f_{\beta}$ and all the maps $\lambda_{n,k,\alpha}$ are piecewise affine and Lebesgue-measure preserving, the map $F_{\beta}$ is piecewise affine and Lebesgue measure-preserving as well and since critical points of $f_{\beta}$ for every $\beta$ are the same as the critical points of $f$, there is a positive real number $\iota$ so that if $a<b\in \mathbb{R}$ are such that $b-a<\iota$ then there is a point $c\in (a,b)$ so that $\tilde{F}_{\beta}$ is linear on both intervals $[a,c]$ and $[c,b]$.
	Since $|\tilde F'_{\beta}(t)|\geq 4$ for $t\in (a,c)\cup (c,b)$ it holds 
	\begin{equation}\label{eq:1}
	\lambda(\tilde{F}_{\beta}([a,b]))\geq 2(b-a) \text{  } \forall a<b\in \mathbb{R} \text{ such that } b-a<\iota.
	\end{equation}

	Furthermore, again since $F_{\beta}$ is piecewise affine and $r_{\beta}$ does not change the slope when composed with $f$ there exists $s\in \mathbb{R}$ so that $|\tilde F_{\beta}'(t)|<s$ for all $t\in \mathbb{R}$ where $|\tilde F_{\beta}(t)|$ exists (actually as the slopes of all the functions in the composition of $F_{\beta}$ are constant, $|\tilde F_{\beta}'(t)|$ is just the product of slopes of corresponding functions in the composition). Thus
\begin{equation}\label{eq:2}
	\lambda(F_{\beta}(C))\leq s\lambda(C) \text{ for every arc } C\subset \mathbb{S}^1.
	\end{equation}

Now let us prove the following claim, which in particular shows that $F_{\beta}$ is leo for every $\beta$.
\begin{clm}\label{clm:uniqueN}
	Let $\xi>0$. There is an integer $N\geq 0$ so that for every arc $J\subset \Ci$ with $\lambda(J)>\xi$ and all $\beta\in\R$ it holds that $F_{\beta}^N(J)=\Ci$.
\end{clm}
\begin{proof}[Proof of Claim~\ref{clm:uniqueN}] Take any $M$ so that $4^M\xi>1$. Since diameter of arcs that contain less than $2$ critical points of $f_{\beta}$ grows at least by a factor $4$ and since $\lambda_{n,k,\alpha}$ by Lemma~\ref{lem:MincUpdt} \eqref{Lcra:3} does not shrink intervals, there is $j<M$ so that $F_{\beta}^j(J)$ contains two critical points of $f_{\beta}$. By the definition of $f_{\beta}$ and application of arguments from the proof of Lemma~\ref{lem:fbetaleo} we have that $F_{\beta}^{j+8}(J)=\Ci$; therefore we can set $N=M+8$.
\end{proof}

 In order to apply Lemma~\ref{lem:MincUpdt}, fix $\eps_i=\frac{n_i-1}{n_i+k_i-1}$ and $\gamma_i=\frac{1}{n_i+k_i-1}$ for every $i\geq 1$.
    We choose odd $n_{m+1}\geq 7$ and even $k_{m+1}\geq 1$ large enough so 
 $\eps_{m+1}<\eta/s$. By Claim~\ref{clm:uniqueN} there exists $N\in \N$ so that if $\lambda(A)>\eps_{m+1}/4$ for some arc $A\subset\mathbb{S}^1$,  then $F_{\beta}^N(A)=\mathbb{S}^1$ for any $\beta\in \R$. 
	Observe that if  $k_{m+1}$ is chosen to be large enough we have
	$$
	0<\gamma_{m+1}<\min (\iota, s^{-N}, \eps_{m+1}/4, \delta s^{-N}/8).
	$$
	It follows from \eqref{eq:1} that:
	\begin{equation}\label{eq:3}
	\lambda(F_{\beta}(A))\geq 2\min\{\lambda(A),\gamma_{m+1}\} \textrm{ for any subarc } A\subset \mathbb{S}^1.
	\end{equation}

	We define $G_{\beta,\alpha,m+1}=F_{\beta}\circ \lambda_{n_{m+1},k_{m+1},\alpha}:\mathbb{S}^1\to \mathbb{S}^1$ and let $\tilde G_{\beta,\alpha,m+1}: \mathbb{R}\to \mathbb{R}$ be a lifting of $G_{\beta,\alpha,m+1}$. Note again that $G_{\beta,\alpha,m+1}$ is a piecewise affine map.  
 Since $|f_\beta'(t)|\geq 4$ for all but finitely many arguments, and it is composed with piecewise linear maps preserving Lebesgue measure, this gives that $|\tilde G_{\beta,\alpha,m+1}'(t)|\geq 4$ for every $t\in \mathbb{R}$ for which $\tilde G_{\beta,\alpha,m+1}'(t)$ exists. 
	Since $\gamma_{m+1}<\eta$ and $|\alpha|<\gamma_{m+1}/2$ we get from Lemma~\ref{lem:MincUpdt}\eqref{Lcra:5} that if
	$\lambda(A)\geq \eta$ for some arc $A\subset\mathbb{S}^1$
	then $A\subset \lambda_{n_{m+1},k_{m+1},\alpha}(A)$ and thus $ F_{\beta}(A)\subset G_{\beta,\alpha,m+1}(A)$.
	
Furthermore, since $|\tilde{\lambda}_{n_{m+1},k_{m+1},\alpha}(t)-t|<\eps_{m+1}/2+\gamma_{m+1}+\alpha<\eps_{m+1}/2+2\gamma_{m+1}<\eps_{m+1}<\eta/s$ and $|\tilde F_{\beta}(\tilde{\lambda}_{n_{m+1},k_{m+1},\alpha}(t))-\tilde F_{\beta}(t) |\leq s|\tilde{\lambda}_{n_{m+1},k_{m+1},\alpha}(t)-t|$, we get $|\tilde{G}_{\beta,\alpha,m+1}(t)-\tilde F_{\beta}(t)|<\eta$ for each $t\in \mathbb{R}$. This implies $\rho(F_{\beta},G_{\beta,\alpha,m+1})<\eta$.

	To show that $G_{\beta,\alpha,m+1}$ is leo, we first prove that for every subarc $C\subset \mathbb{S}^1$ such that $\lambda(C)>\eps_{m+1}/4$, it follows that 
		\begin{equation}\label{eq:4}
	G_{\beta,\alpha,m+1}^{N-1}(f_{\beta}(C))=\mathbb{S}^1
	\end{equation}
	 where $N$ is provided by Claim~\ref{clm:uniqueN} to $\xi=\eps_{m+1}/4$.
	
		By Claim~\ref{clm:uniqueN}, since $f_{\beta}^N(C)=\mathbb{S}^1$ and $\gamma_{m+1}\leq s^{-N}$, it follows from \eqref{eq:2} that for each $j\in \N$ $\lambda(f_{\beta}^j(C))\geq \gamma_{m+1}$. Using Lemma~\ref{lem:MincUpdt}~\eqref{Lcra:5} inductively we get $f_{\beta}^j(C)\subset G_{\beta,\alpha,m+1}^{j-1}({f_{\beta}}(C))$ for $j\in \{1,\ldots, N\}$ which proves condition \eqref{eq:4}.
	
	Let $A\subset \mathbb{S}^1$ be an arc with $\lambda(A)>\gamma_{m+1}$.  By Lemma~\ref{lem:MincUpdt}~\eqref{Lcra:4}  it follows
	there exists $c\in\Int(A)$ and arcs $A_1,A_2\subset A$ such that $A=A_1\cup A_2$ and $A_1\cap A_2=\{c\}$ with $\lambda(\lambda_{n_{m+1},k_{m+1},\alpha}(A_1))>\eps_{m+1}/4$ and $\lambda(\lambda_{n_{m+1},k_{m+1},\alpha}(A_2))>\eps_{m+1}/4$. Lemma~\ref{lem:MincUpdt}~\eqref{Lcra:4} actually gives the statement that $\lambda(\lambda_{n_{m+1},k_{m+1},\alpha}(A))>\eps_{m+1}/2$ 
which leads by Darboux property to desired conditions. But then we have the following:
		\begin{equation}\label{eq:5}
		\begin{split}
		&\quad \text{if } \lambda(A)>\gamma_{m+1} \text{ then  there exists } c\in\Int(A) \text{ and arcs } A_1, A_2\subset A \text{ such that }\\
		&\quad A=A_1\cup A_2,  A_1\cap A_2=\{c\} \text{ for which } G_{\beta,\alpha,m+1}^N(A_1)=\mathbb{S}^1 \text{ and }\\
		&\quad 
		 G_{\beta,\alpha,m+1}^N(A_2)=\mathbb{S}^1.
		\end{split}
		\end{equation}
		 In particular $G_{\beta,\alpha,m+1}^N(A)=\mathbb{S}^1$.

	To finish the proof that $G_{\beta,\alpha,m+1}$ is leo we  need to show that any arc $A\subset \mathbb{S}^1$ will get mapped to $\mathbb{S}^1$ by some iteration of $G_{\beta,\alpha,m+1}$. But it follows from \eqref{eq:1}, Lemma~\ref{lem:MincUpdt}\eqref{Lcra:3} and the fact that $F_{\beta}$ is Lebesgue measure-preserving that there exists $m\in \N$ such that $\lambda(G_{\beta,\alpha,m+1}^{m}(A))\geq \gamma_{m+1}$. Using \eqref{eq:5} we indeed get $G_{\beta,\alpha,m+1}^{m+N}(A)=\mathbb{S}^1$. Moreover, since the slopes of all the functions in the composition of $G_{\beta,\alpha,m+1}$ are constant, the absolute value of $\tilde G_{\beta,\alpha,m+1}'$ wherever defined is just the product of slopes (wherever defined) of corresponding functions in the composition; the absolute value of the slope of  $f_{\beta}$ is $13$ and the absolute values of slopes of $\lambda_{n,k,\alpha}$ for $n\geq 7$ and $k\geq 1$ is also strictly greater than $1$, therefore $|\tilde G_{\beta,\alpha,m+1}'(t)|>4$ for all $t\in \R$ where $\tilde G_{\beta,\alpha,m+1}'$ is defined. This shows that maps $\tilde G_{\beta,\alpha,m+1}$ are admissible for every $\beta\in\R$. Obviously, every $G_{\beta,\alpha,m+1}$ is also Lebesgue measure-preserving degree one circle map as a composition of Lebesgue measure-preserving degree one circle maps.

	What remains to be shown is the statement about $\delta$-crookedness.
	First we have a claim that will be important in the proof of $\delta$-crookedness.
	 
	\begin{clm}\label{clm:1}
	Let $A\subset \mathbb{S}^1$ be a subarc such that $\lambda(A)\geq \gamma_{m+1}$, $B\subset \mathbb{S}^1$ and let $r\in \mathbb{R}$ be such that $B\subset B(A,r)$.  
	For each $j\in \N$, $G_{\beta,\alpha,m+1}^j(B)\subset B(G_{\beta,\alpha,m+1}^j(A),q)$ where $q= s^j(r+2\gamma_{m+1})$.
	\end{clm}

	\begin{proof}[Proof of Claim~\ref{clm:1}]
	Using \eqref{eq:3} and Lemma~\ref{lem:MincUpdt}~\eqref{Lcra:3} we get that $\lambda(G_{\beta,\alpha,m+1}^{i}(A))\geq \gamma_{m+1}$ for each $i\in \N$. Now set $q_0=r$, $q_{i+1}=s(q_i+\gamma_{m+1})$ for $i\in \N_0$. Repeatedly using Lemma~\ref{lem:MincUpdt}~\eqref{Lcra:6} and \eqref{eq:2} one gets $G_{\beta,\alpha,m+1}^{j}(B)\subset B( G_{\beta,\alpha,m+1}^{j}(A),q_j)$. Observing $q_j=s^j r +\gamma_{m+1}\frac{s^{j+1}-1}{s-1}<s^j(r+\frac{s}{s-1}\gamma_{m+1})<s^j(r+2\gamma_{m+1})$ finishes the claim.
	\end{proof}

		Take $\omega:[0,1]\to \mathbb{S}^1$; we will show that it is $(G_{\beta,\alpha,m+1}^{N},\delta)$-crooked.
 Put $A=\omega([0,1])$ and 
$$
[a,b]:=[G_{\beta,\alpha,m+1}^N(0),G_{\beta,\alpha,m+1}^N(1)]\subset G_{\beta,\alpha,m+1}^N(A).
$$
First, we can obviously assume that $b-a>2\delta$. 
To prove that $\omega$ is $(G_{\beta,\alpha,m+1}^{N},\delta)$-crooked
		 it is enough to show:
 
		\begin{equation}\label{eq:7}
			\exists \text{ two disjoint subarcs } A_1,A_2\subset \mathbb{S}^1 \text{ so that } [a+\delta,b-\delta] \subset  G_{\beta,\alpha,m+1}^{N}(A_1)\cap G_{\beta,\alpha,m+1}^N(A_2).
		\end{equation}

		For $\lambda(A)\geq \gamma_{m+1}$ statement of \eqref{eq:5} directly implies \eqref{eq:7}; therefore we can assume that $\lambda(A)<\gamma_{m+1}$. First note that $\lambda(G^{N}_{\beta,\alpha,m+1}(A))\geq 2\delta>\gamma_{m+1}$. Let $M\in\N$ be the largest natural number so that $\lambda(G^{M}_{\beta,\alpha,m+1}(A))<\gamma_{m+1}$. 
			Let $[c,d]:=\lambda_{n_{m+1},k_{m+1},\alpha}( G^{M}_{\beta,\alpha,m+1}(A))$. Since this arc does not cover the circle, we can assume that in a natural ordering on $[c,d]$ of $\mathbb{S}^1$ we have $c<d$.
   
		{\bf Case I.} $\lambda(\lambda_{n_{m+1},k_{m+1},\alpha}(G^{M}_{\beta,\alpha,m+1}(A)))>4\gamma_{m+1}$.\\
	 Because $\lambda(G^{M}_{\beta,\alpha,m+1}(A))<\gamma_{m+1}$, Lemma~\ref{lem:MincUpdt}~\eqref{Lcra:1} gives that 
  \begin{eqnarray*}
  d-c&<&\lambda(G^{M}_{\beta,\alpha,m+1}(A)+2\rho(\lambda_{n_{m+1},k_{m+1},\alpha},\text{id}_{\mathbb{S}^1})\\
  &< &\eps_{m+1}+3\gamma_{m+1}+2\alpha\leq \eps_{m+1}+4\gamma_{m+1}.
  \end{eqnarray*}
  There exist points $c_1<d_1\in [c,d]$ so that $4\gamma_{m+1}<d_1-c_1<\eps_{m+1}$ and $[c,d]\subset B([c_1,d_1],4\gamma_{m+1})$. 

 Let $[a_1,b_1]\subset [a,b]$ be any interval such that $\lambda_{n_{m+1},k_{m+1},\alpha}([a_1,b_1])=[c_1,d_1]$.
Lemma~\ref{lem:MincUpdt}~\eqref{Lcra:2} gives that any map $\omega\colon [0,1]\to [a_1,b_1]$ is $(\lambda_{n_{m+1},k_{m+1},\alpha},3\gamma_{m+1})$-crooked, therefore
there exist two disjoint subarcs $M_1,M_2\subset G^{M}_{\beta,\alpha,m+1}(A)$ such that $\lambda_{n_{m+1},k_{m+1},\alpha}(M_1)=[c_1,d_1-3\gamma_{m+1}]$ and $\lambda_{n_{m+1},k_{m+1},\alpha}(M_2)=[c_1+3\gamma_{m+1},d_1]$.

There exist two disjoint subarcs $A_1,A_2\subset A$ so that $G^{M}_{\beta,\alpha,m+1}(A_1)=M_1$ and  $G^{M}_{\beta,\alpha,m+1}(A_2)=M_2$. Observe that 
$$\lambda(\lambda_{n_{m+1},k_{m+1},\alpha}(G^{M}_{\beta,\alpha,m+1}(A))\subset B(\lambda_{n_{m+1},k_{m+1},\alpha}(M_1),7\gamma_{m+1})$$ as well as
$$
\lambda(\lambda_{n_{m+1},k_{m+1},\alpha}(G^{M}_{\beta,\alpha,m+1}(A))\subset B(\lambda_{n_{m+1},k_{m+1},\alpha}(M_2)),7\gamma_{m+1})
$$ 
and it follows from \eqref{eq:2} that $G^{M+1}_{\beta,\alpha,m+1}(A)\subset B(G_{\beta,\alpha,m+1}(M_1),7s\gamma_{m+1})$ as well as $G^{M+1}_{\beta,\alpha,m+1}(A)\subset B(G_{\beta,\alpha,m+1}(M_2),7s\gamma_{m+1})$. 

Equation \eqref{eq:3} implies that that $\lambda(G_{\beta,\alpha,m+1}(M_1))>\gamma_{m+1}$, $\lambda(G_{\beta,\alpha,m+1}(M_2))>\gamma_{m+1}$ which by Claim~\ref{clm:1} gives that 
		$$G^{M+N}_{\beta,\alpha,m+1}(A)=G^{N-1}_{\beta,\alpha,m+1}(G^{M+1}_{\beta,\alpha,m+1}(A))\subset B(G_{\beta,\alpha,m+1}^{N-1}(M_1),s^{N-1}(7s\gamma_{m+1}+2\gamma_{m+1})),$$
		 and
		 $$G^{M+N}_{\beta,\alpha,m+1}(A)=G^{N-1}_{\beta,\alpha,m+1}(G^{M+1}_{\beta,\alpha,m+1}(A))\subset B(G_{\beta,\alpha,m+1}^{N-1}(M_2),s^{N-1}(7s\gamma_{m+1}+2\gamma_{m+1})).$$
		 Because $s^{N-1}(7s\gamma_{m+1}+2\gamma_{m+1})<8\gamma_{m+1}s^{N}<\delta$ it follows that $$[a+\delta,b-\delta]\subset G^{N-1}_{\beta,\alpha,m+1}(M_1)\cap G^{N-1}_{\beta,\alpha,m+1}(M_2)=G^{N+M}_{\beta,\alpha,m+1}(A_1)\cap G^{N+M}_{\beta,\alpha,m+1}(A_2).$$
{\bf Case II.} $\lambda(\lambda_{n_{m+1},k_{m+1},\alpha}(G^{M}_{\beta,\alpha,m+1}(A)))\leq 4\gamma_{m+1}$.\\
			Using \eqref{eq:2} and the choice of $M$ we have $\gamma_{m+1}\leq \lambda(G^{M}_{\beta,\alpha,m+1}(A))\leq 4s\gamma_{m+1}$. By Lemma~\ref{lem:MincUpdt}~\eqref{Lcra:1} and \eqref{Lcra:4} we get $\eps_{m+1}/2<d-c<\eps_{m+1}+\gamma_{m+1}(4s+3)$.
   There exist points $c_1<d_1\in [c,d]$ so that $4\gamma_{m+1}<d_1-c_1<\eps_{m+1}$ and $[c,d]\subset B([c_1,d_1],\gamma_{m+1}(4s+3))$.
			Lemma~\ref{lem:MincUpdt}~\eqref{Lcra:2} gives that a map $\omega:[0,1]\to [c_1,d_1]$ is $(\lambda_{n_{m+1},k_{m+1},\alpha},3\gamma_{m+1})$-crooked. Therefore, there exist two disjoint subarcs $M_1,M_2\subset G^{M+1}_{\beta,\alpha,m+1}(A)$ such that $\lambda_{n_{m+1},k_{m+1},\alpha}(M_1)=[c_1,d_1-3\gamma_{m+1}]$ and $\lambda_{n_{m+1},k_{m+1},\alpha}(M_2)=[c_1+3\gamma_{m+1},d_1]$. There exist two disjoint subarcs $A_1,A_2\subset A$ so that $G^{M+1}_{\beta,\alpha,m+1}(A_1)=M_1$ and  $G^{M+1}_{\beta,\alpha,m+1}(A_2)=M_2$.

		Since $s\geq 4$ we get $\lambda(\lambda_{n_{m+1},k_{m+1},\alpha}(G^{M+1}_{\beta,\alpha,m+1}(A))\subset B(\lambda_{n_{m+1},k_{m+1},\alpha}(M_1),
    \gamma_{m+1}(4s+3)+3\gamma_{m+1})\subset B(\lambda_{n_{m+1},k_{m+1},\alpha}(M_1),5s\gamma_{m+1})$ as well as
		$\lambda(\lambda_{n_{m+1},k_{m+1},\alpha}(G^{M+1}_{\beta,\alpha,m+1}(A))\subset B(\lambda_{n_{m+1},k_{m+1},\alpha}(M_2),\gamma_{m+1}( 4s+3)+3\gamma_{m+1})\subset B(\lambda_{n_{m+1},k_{m+1},\alpha}(M_2),5s\gamma_{m+1})$ and it follows from \eqref{eq:2} that $G^{M+2}_{\beta,\alpha,m+1}(A)\subset B(G_{\beta,\alpha,m+1}(M_1),5s^2\gamma_{m+1})$ as well as $G^{M+2}_{\beta,\alpha,m+1}(A)\subset B(G_{\beta,\alpha,m+1}(M_2), 5s^2\gamma_{m+1})$. By equation \eqref{eq:3} we get\\$\lambda(G_{\beta,\alpha,m+1}(M_1)),\lambda(G_{\beta,
        \alpha,m+1}(M_2))>\gamma_{m+1}$. It follows from Claim~\ref{clm:1} that 
		
		$$G^{N+M}_{\beta,\alpha,m+1}(A)=G^{N-2}_{\beta,\alpha,m+1}(G^{M+2}_{\beta,\alpha,m+1}(A))\subset B(G_{\beta,\alpha,m+1}^{N-2}(M_1),s^{N-2}(5s^2\gamma_{m+1}+2\gamma_{m+1})),$$
		and
		$$G^{N+M}_{\beta,\alpha,m+1}(A)=G^{N-2}_{\beta,\alpha,m+1}(G^{M+2}_{\beta,\alpha,m+1}(A))\subset B(G_{\beta,\alpha,m+1}^{N-2}(M_2),s^{N-2}(5s^2\gamma_{m+1}+2\gamma_{m+1})).$$

		Because $s^{N-2}(5s^2\gamma_{m+1}+2\gamma_{m+1})<6\gamma_{m+1}s^{N}<\delta$ it follows that $$[a+\delta,b-\delta]\subset G^{N-2}_{\beta,\alpha,m+1}(M_1)\cap G^{N-2}_{\beta,\alpha,m+1}(M_2)=G^{N+M}_{\beta,\alpha,m+1}(A_1)\cap G^{N+M}_{\beta,\alpha,m+1}(A_2).$$
	
\end{proof}

\begin{lem}\label{lem:symm-half}
For every $\alpha,\beta\in \R$ it follows $\tilde G_{\beta,\alpha,m+1}(t)+1/2=\tilde G_{\beta,\alpha,m+1}(t+1/2)$ for every $t\in \R$.
\end{lem}	

\begin{proof}
	First note that due to the repetitive structure of $f_{\beta}$ it holds $f_{\beta}(t)+1/2=f_{\beta}(t+1/2)$ for every $t\in \mathbb{S}^1$ and every $\beta\in \R$. Furthermore, because of the repetitive block structure of maps $\lambda_{n_{m+1},k_{m+1},\alpha}$ if $n_{m+1}$ is odd and $k_{m+1}$ is even then also $\lambda_{n_{m+1},k_{m+1},\alpha}(t)+1/2=\lambda_{n_{m+1},k_{m+1},\alpha}(t+1/2)$ for every $t\in \mathbb{S}^1$ by Observation~\ref{obs:gammaapart2}. Therefore, we first obtain that $F_{\beta}(t)+1/2=F_{\beta}(t+1/2)$ for every $t\in \Ci$ and every $\beta\in  \R$ and thus also $G_{\beta,\alpha,m+1}(t)+1/2=G_{\beta,\alpha,m+1}(t+1/2)$ for every $t\in \Ci$, every $\beta\in  \R$ and every $|\alpha|<\frac{1}{2(n_{m+1}+k_{m+1}-1)}$. But then the same follows for the liftings, namely $\tilde G_{\beta,\alpha,m+1}(t)+1/2=\tilde G_{\beta,\alpha,m+1}(t+1/2)$ for every $t\in \R$ and every $|\alpha|<\frac{1}{2(n_{m+1}+k_{m+1}-1)}$.
\end{proof}

\begin{lem}\label{lem:pi}
For all $i\in\{1,\ldots, m\}$ and arbitrary $\beta\in \R$ let integers $n_i,k_i$, and real numbers $\alpha_i$ and the map $F_{\beta}$ be fixed. Let $x\in \mathbb{R}$ so that $\tilde F_{\beta}(x)>\tilde F_{\beta}(y)$ for each $y<x$. Let  $n_{m+1}$ and $k_{m+1}$ be provided by Lemma~\ref{lem:LemMTadjusted} for some $\eta$ and $\delta$.  Then there are $|\alpha_{m+1}|\leq \frac{1}{2(n_{m+1}+k_{m+1}-1)}$ and $p_{m+1}\in\mathbb{R}$ 
such that $\tilde G_{\beta,\alpha_{m+1},m+1}(p_{m+1})> \tilde G_{\beta,\alpha_{m+1},m+1}(q)$ for all $q<p_{m+1}$ and $\tilde\lambda_{n_{m+1},k_{m+1},\alpha_{m+1}}(p_{m+1})=x$.
In particular $|p_{m+1}-x|\leq \frac{n_{m+1}}{n_{m+1}+k_{m+1}-1}$.
\end{lem}
\begin{proof}	

	By Observation~\ref{obs:gammaapart} 
 there are $x_j,x_{j+1}\in \mathbb{R}$ such that $|\tilde\lambda_{n_{m+1},k_{m+1},0}(x_j)-\tilde\lambda_{n_{m+1},k_{m+1},0}(x_{j+1})|=\frac{1}{n_{m+1}+k_{m+1}-1}$ and  $\tilde\lambda_{n_{m+1},k_{m+1},0}(x_j)\leq x\leq \tilde\lambda_{n_{m+1},k_{m+1},0}(x_{j+1})$. 
 Then there is $|\alpha_{m+1}|\leq \frac{1}{2(n_{m+1}+k_{m+1}-1)}=\gamma_{m+1}/2$
 such that $\tilde\lambda_{n_{m+1},k_{m+1},\alpha_{m+1}}(r_{-\alpha_{m+1}}(x_j))=x$ or $\tilde\lambda_{n_{m+1},k_{m+1},\alpha_{m+1}}(r_{-\alpha_{m+1}}(x_{j+1}))=x$. Put $p_{m+1}=r_{-\alpha_{m+1}}(x_j)$ or $p_{m+1}=r_{-\alpha_{m+1}}(x_{j+1})$ depending on the case. By the definition of $x_j$ (resp. $x_{j+1}$) we have that
 $\tilde\lambda_{n_{m+1},k_{m+1},\alpha_{m+1}}(q)<\tilde\lambda_{n_{m+1},k_{m+1},\alpha_{m+1}}(p_{m+1})=x$
 for every $q<p_{m+1}$, and so composing with $F_{\beta}$ gives 
 $\tilde G_{\beta,\alpha_{m+1},m+1}(p_{m+1})> \tilde G_{\beta,\alpha_{m+1},m+1}(q)$ for all $q<p_{m+1}$.

The proof is completed by Lemma~\ref{lem:MincUpdt}\eqref{Lcra:1}, because
$$
|p_{m+1}-x|\leq  \rho(\tilde\lambda_{n_{m+1},k_{m+1},\alpha_{m+1}},\text{id}_{\R})<\eps_{m+1}/2+\gamma_{m+1}+\alpha_{m+1}\leq \frac{n_{m+1}}{n_{m+1}+k_{m+1}-1}.
$$

\end{proof}

Define $G_{\beta}:\mathbb{S}^1\to \mathbb{S}^1$ by $G_{\beta}:=\lim_{m\to\infty}G_{\beta,\alpha_{m+1},m+1}$, where $\alpha_{m+1}$
are recursively provided by Lemma~\ref{lem:pi} together with points $p_m$ such that
$\lambda_{n_{m+1},k_{m+1},\alpha_{m+1}}(p_{m+1})=p_m$ and $f_\beta(p_0)$
is such that $\tilde f_\beta(x)<\tilde f_\beta(p_0)$ for every $x<p_0$.

Note that in Lemma~\ref{lem:pi} we may freely assign $\eta$ and $\delta$.
In particular, we can take $\eta$
so small, that for all $q<p_m-\frac{1}{2(m+1)}$ we have
$|F_\beta(p_m)-F_\beta(q)|>4\eta$. Since we may also require that $\rho(\lambda_{n_{m+1},k_{m+1},\alpha_{m+1}},\text{id})<\frac{1}{4(m+1)}$, we obtain that 
$
|G_{\beta,\alpha_{m+1},m+1}(p_{m+1})-G_{\beta,\alpha_{m+1},m+1}(q)|>3\eta$ for all $q<p_{m+1}-\frac{1}{m+1}$.
Denote by $\eta_{m+1}$ the constant associated with the construction of $\tilde G_{\beta,\alpha_{m+1},m+1}$.

\begin{lem}\label{lem:limitGbeta}
The limit maps $G_{\beta}$ are well defined and $\underleftarrow{\lim}(\mathbb{S}^1,G_{\beta})$ is the pseudo-circle for every $\beta\in \R$. Furthermore, the following statements hold:
\begin{enumerate}
 \item\label{item1} $\tilde G_{\beta}(t+1/2)=\tilde G_{\beta}(t)+1/2$ for all $t\in \mathbb{R}$.
 \item\label{item2} there exists $p\in\mathbb{R}$ such that $\tilde G_{\beta}(p)\geq \tilde G_{\beta}(q)$ for all $q<p$ and for all $\beta\in \R$.
  Furthermore, if the sequence $\eta_m$ converges sufficiently fast to zero, then  $\tilde G_{\beta}(p)> \tilde G_{\beta}(q)$ for all $q<p$ and for all $\beta\in \R$.
 \end{enumerate}
\end{lem}

\begin{proof}
First observe that we can take a sequence of maps $G_{\beta,\alpha_{m+1},m+1}$ such that for each positive integer $i$ we have $\rho(G_{\beta,\alpha_{i},i},G_{\beta,\alpha_{i+1},i+1})<2^{-i}$. Thus it follows that the sequence of maps $G_{\beta,\alpha_{m+1},m+1}$ converges uniformly and thus $G_{\beta}$ exists. Furthermore, we can assume that maps $G_{\beta,\alpha_{m+1},m+1}$ and sequences $n(1),n(2),\ldots $ are such that $G^{n(k)}_{\beta,\alpha_{i},i}$ is $(2^{-k}-2^{-k-i})$-crooked for each positive integer $k\leq i$ and furthermore that every $G_{\beta,\alpha_{i},i}$ is admissible (this follows already from the construction).
Observe that provided $(f_i)_{i\geq 1}$ is a sequence of circle maps that converge uniformly to $f$ and are all at most $\delta$-crooked, then $f$ is also $\delta$-crooked.
Now applying Proposition~\ref{prop:pseudo-circle}, Lemma~\ref{lem:crooked} and the last statement we get that $G^{n(k)}_{\beta}$ is $2^{-k}$-crooked for each positive integer $k$ which shows that $\underleftarrow{\lim}(\mathbb{S}^1,G_{\beta})$ is the pseudo-circle for every $\beta\in \R$.
 The statement \eqref{item1} follows directly from Lemma~\ref{lem:symm-half}.

By Lemma~\ref{lem:pi} the sequence $p_m$ is convergent and $p:=\lim_{m\to\infty}p_m$ exists.
Fix any $q<p$ and take $M\in \N$ such that
$p_m>q$ for every $m>M$.
Then by definition of $p_m$ we have that $\tilde G_{\beta,\alpha_{m},m}(q)<\tilde G_{\beta,\alpha_{m},m}(p_m)$ and so passing to the limit we get
$\tilde G_{\beta}(q)\leq \tilde G_{\beta}(p)$. Let us assume towards the contradiction that 
$\tilde G_{\beta}(q)= \tilde G_{\beta}(p)$. Suppose wlog that $|q-p|>4/M$. Then $p_m-1/m>q$ for every $m>M$
and if $\sum_{m\geq M}\eta_m<2\eta_M$
then
$$
|G_{\beta,\alpha_M,M}(p_M)-G_{\beta,\alpha_M,M}(q)|\geq
|G_{\beta,\alpha_m,m}(p_{m})-G_{\beta,\alpha_M,M}(q)|>3\eta_M
$$
while by telescoping we get
$$
|G_{\beta,\alpha_M,M}(q)-G_{\beta,\alpha_m,m}(q)|\leq \sum_{m\geq M}\eta_m<2\eta_M.
$$
Finally we get that $|G_{\beta,\alpha_m,m}(p_m)-G_{\beta,\alpha_m,m}(q)|
>\eta_M$.  
This completes the proof.
\end{proof}

\begin{lem}\label{lem:periodic} If $\beta=p+1/2-G_0(p)$
then $G_\beta(p)=p+1/2$.
\end{lem}
\begin{proof}
By the definition of $f_\beta$ we obtain that $G_\beta=r_\beta \circ G_0$. But this immediately gives
$$
G_\beta(p)=G_0(p)+\beta=p+1/2
$$
completing the proof.
\end{proof}

{

\begin{thm}\label{thm:main}
There exists a parameterized family of maps $\{G_{\beta}\}_{\beta\in\R}\subset \mathcal{T}\subset C_{\lambda}(\mathbb{S}^1)$ varying continuously  with $\beta$ such that:
\begin{enumerate}
\item\label{main1} The inverse limit $\underleftarrow{\lim} (\mathbb{S}^1,G_{\beta})$ is the pseudo-circle for every $\beta$.
 \item\label{main2} $\tilde G_{\beta}(t+1/2)=\tilde G_{\beta}(t)+1/2$ for all $t\in \mathbb{R}$.
 \item\label{main3} there exists $p\in\mathbb{R}$ such that $\tilde G_{\beta}(p)> \tilde G_{\beta}(q)$ for all $q<p$ and for all $\beta$.
  \item\label{main4} $G_{\beta}$ is leo.
 \end{enumerate}
\end{thm}
\begin{proof}
Since the family of maps $\{f_{\beta}\}_{\beta\in\R}$ varies continuously, and  we compose all of them with the same family of perturbations, continuity of family $\{G_\beta\}_{\beta\in \R}$ with respect to the parameter $\beta$ is obvious. We get the result if we combine Lemma~\ref{lem:limitGbeta} with the proof of Theorem~\ref{thm:UniLimPresLeb}; namely, to get crookedness and admissibility on every step we repeatedly use Lemma~\ref{lem:LemMTadjusted}.
Recall that for any $k\geq 0$ the set $A_k\subset C_\lambda(I)$ is contained in the set of maps $f$ such that $f^{s}$ is $(1/k-\delta)$-crooked for some $s$ and some sufficiently small $\delta>0$. Starting with $\{f_{\beta}\}_{\beta\in\R}$ we use Lemma~\ref{lem:LemMTadjusted} directly to obtain maps $\{{G}_{\beta,\alpha_{m},m}\}_{\beta\in \R}\subset A_1$. But if $\{{G}_{\beta,\alpha_{m},m}\}_{\beta\in \R}\subset A_1$ and $m,\delta$ are constants from the definition of $A_1$, then by Lemma~\ref{lem:crooked} we have $B(\{{G}_{\beta,\alpha_{m},m}\}_{\beta\in \R},\delta/4)\subset A_1$. For the second step we take the family $\{{G}_{\beta,\alpha_{m},m}\}_{\beta\in \R}$. Proceeding as in the rest of the proof of Theorem~\ref{thm:UniLimPresLeb}, ensuring sufficiently fast convergence, we obtain in the intersection of sets $A_k$ the family $\{G_{\beta}\}_{\beta\in \R}\subset \mathcal{T}$ of continuous maps varying with $t$.\\
To get \eqref{main4} first note that it suffices to show $G_{\beta}$ is transitive since it was shown in \cite[Corollary 21]{KOT} that every transitive circle map that gives pseudo-circle in the inverse limit (as single bonding map) is actually leo.

Fix $m$ and note that by Lemma~\ref{lem:LemMTadjusted} each map $\{{G}_{\beta,\alpha_{m},m}\}_{\beta\in \R}$ is leo for any $\beta\in \R$.
For a fixed $\beta$ and any interval $J$ of length at least $1/m$
there is $k$ such that ${G}_{\beta,\alpha_{m},m}^k(J)=\mathbb{S}^1$. In fact, if we check the proof of  Lemma~\ref{lem:LemMTadjusted} (see Claim~\ref{clm:uniqueN} and the arguments thereafter), we see that this $k$ depends only on slopes of ${G}_{\beta,\alpha_{m},m}$ and the number of pieces of monotonicity rather than $\beta$, that is we can find one $k$ good for all maps ${G}_{\beta,\alpha_{m},m}$ and all intervals $J$. But then, if $g$ is sufficiently small perturbation of any  ${G}_{\beta,\alpha_{m},m}$ then $d_H(g^k(J),{G}_{\beta,\alpha_{m},m}^k(J))<1/m$, meaning $g^k$ covers whole $\mathbb{S}^1$ except possibly an interval of length at most $1/m$. So transitivity is ensured by sufficiently fast convergence of ${G}_{\beta,\alpha_{m},m}$ to $G_\beta$, which we can control in the construction process of ${G}_{\beta,\alpha_{m+1},m+1}$ from ${G}_{\beta,\alpha_{m},m}$.
\end{proof}

Now let us briefly describe the BBM construction for the family of circle maps $\{G_{\beta}\}_{\beta\in\R}$. 
  Let us consider the annulus $\mathbb{A}=[0,1)\times [-k-3,k+3]$ where $k>\sup_{t\in [0,1)} G_{\beta}(t)$ and $-k>\inf_{t\in [0,1)} G_{\beta}(t)$ (see Figure~\ref{fig:lifting}). We will describe the particular way in which we decompose $[0,1)\times [-k-1,k+1]$ into a family of continuously varying arcs and how maps $G_{\beta}$ extend to homeomorphisms $\bar{G}_{\beta}$ on $[0,1)\times [-k-1,k+1]$. 
  We partition $\mathbb{A}$ into lines of slope $1$;  $\bar G_{\beta}$ will preserve these lines. For each such line we will define $\bar G_{\beta}$ on finitely many points and extend linearly between them on each such line. First of all we put $\bar G_{\beta}(x,k+1)=(x,k+1)$ and $\bar G_{\beta}(x,-k-1)=(x,-k-1)$, that is, at start we keep the map identity on the boundary. Later in the construction we will change the map on the boundary to a rotation.
We map points of the middle circle in the following way:
$$(x,0)\mapsto (G_{\beta}(x)\pmod 1, G_{\beta}(x)-x), \forall x\in [0,1),$$
and also define it on vertical lines
$$
(p,t) \mapsto
\begin{cases}
(G_{\beta}(p)\pmod 1, G_{\beta}(p)-p+t); t\in [0,k-G_{\beta}(p)+p],\\
(G_{\beta}(p)-s(G_{\beta}(p)-p) \pmod 1,k+s); t\in  [k-G_{\beta}(p)+p,1],
\end{cases}
$$
where $s=\frac{t-k+G_{\beta}(p)-p}{1+G_{\beta}(p)-p}$ and

$$
\Big(p+\frac{1}{2},t \Big) \mapsto \Big(p_1(\bar{G}_{\beta}(p,t))+\frac{1}{2} \pmod 1,p_2(\bar{G}_{\beta}(p,t))\Big),
$$
where $p_1$ and $p_2$ are natural coordinate projections to $[0,1)$ and $[-k-1,k+1]$ respectively. Note that each line from the partition has finite intersection with sets where $\bar G_{\beta}$ is defined up to now, see Figure~\ref{fig:lifting}.
Now extend the map in the linear fashion on the rest of $[0,1)\times [-k-1,k+1]$.\\
First note that the images on second coordinates from the definition above are always smaller than $k$ and that $G_{\beta}(p)-p+t=k+s$ for $t=k-G_{\beta}(p)+p$ and thus $\bar{G}_{\beta}$ is well defined and continuous for every $\beta\in \R$. 
Let us perform the last modification on the values of points under $\bar{G}_{\beta}$. 
Let us take $(x,t)$ in the range of $\bar{G}_{\beta}$. If $t<k$ then we keep the same value and if $t=k+s$ for some $s\in [0,1]$, then we rotate $x$ by $s(G_{\beta}(p)-p)$. By this modification the red lines on Figure~\ref{fig:lifting} became straight vertical lines.
At the same time the image of $[0,1)\times\{0\}$ under $\bar{G}_{\beta}$ is not affected. Now it is time to extend $\bar G_{\beta}$ to whole $\mathbb{A}$. We will put $\bar{G}_{\beta}(x,t)=(x+G_{\beta}(p)-p \pmod 1,t)$ 
for $t>k+1$ and $\bar{G}_{\beta}(x,t)=(x,t)$ 
for $t<-k-1$.

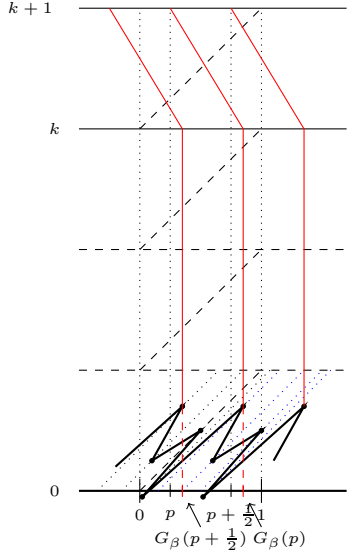
\begin{figure}[ht]
\centering
\begin{tikzpicture}[scale=1.6]
\draw[thick](-0.5,0)--(1.7,0);
\draw[dashed](-0.5,1)--(1.7,1);
\draw[dashed](-0.5,2)--(1.7,2);
\draw(-0.5,3)--(1.7,3);
\draw(-0.5,4)--(1.7,4);
\draw(0,-0.1)--(0,0.1);
\draw(1,-0.1)--(1,0.1);
\draw(0.25,-0.05)--(0.25,0.05);
\draw(0.75,-0.05)--(0.75,0.05);
\draw[dotted](0,0)--(0,4);
\draw[dotted](1,0)--(1,4);
\draw[dotted](0.25,0)--(0.25,4);
\draw[dotted](0.75,0)--(0.75,4);
\node at (0,-0.2){\tiny $0$};
\node at (1,-0.2){\tiny $1$};
\node at (-0.7,3){\tiny $k$};
\node at (-0.7,0){\tiny $0$};
\node at (-0.9,4){\tiny $k+1$};
\node at (0.25,-0.2){\tiny $p$};
\node at (0.75,-0.2){\tiny $p+\frac{1}{2}$};
\draw[dashed](0,0)--(1,1);
\draw[dashed](0,1)--(1,2);
\draw[dashed](0,2)--(1,3);
\draw[dashed](0,3)--(1,4);
\draw[fill=black] (0.35,0.7) circle (.12ex);
\draw[fill=black] (0.85,0.7) circle (.12ex);
\draw[fill=black] (0.1,0.25) circle (.12ex);
\draw[fill=black] (0.02,-0.05) circle (.12ex);
\draw[fill=black] (0.5,0.5) circle (.12ex);
\draw[fill=black] (1,0.5) circle (.12ex);
\draw[fill=black] (0.6,0.25) circle (.12ex);
\draw[fill=black] (0.52,-0.05) circle (.12ex);
\draw[fill=black] (1.35,0.7) circle (.12ex);
\draw[dotted](-0.35,0)--(0.65,1);
\draw[dotted,blue](0.15,0)--(1.15,1);
\draw[dotted](-0.15,0)--(0.85,1);
\draw[dotted,blue](0.35,0)--(1.35,1);
\draw[dotted](0.02,-0.05)--(1.07,1);
\draw[dotted, blue](0.52,-0.05)--(1.57,1);
\draw[dotted,blue](0.5,0)--(1.5,1);
\draw[thick](-0.2,0.2)--(0.35,0.7)--(0.1,0.25)--(0.5,0.5)--(0.02,-0.05)--(0.85,0.7)--(0.6,0.25)--(1,0.5)--(0.52,-0.05)--(1.35,0.7)--(1.1,0.25);
\draw[red](0.85,0.7)--(0.85,3)--(0.25,4);
\draw[red](0.35,0.7)--(0.35,3)--(-0.25,4);
\draw[red](1.35,0.7)--(1.35,3)--(0.75,4);
\draw[red,dashed](0.35,0.7)--(0.35,0);
\draw[red,dashed](0.85,0.7)--(0.85,0);
\draw[red](0.35,-0.05)--(0.35,0.05);
\draw[red](0.85,-0.05)--(0.85,0.05);
\draw[->] (0.5,-0.3)--(0.4,-0.1);
\node at (0.5,-0.4){\tiny $G_{\beta}(p+\frac{1}{2})$};
\draw[->] (1,-0.3)--(0.9,-0.1);
\node at (1.15,-0.41){\tiny $G_{\beta}(p)$};
\end{tikzpicture}
\caption{Sketch of the graph of $\bar{G}_{\beta}$ in the first step of the construction.}\label{fig:lifting}
\end{figure}

We can then associate a {\em retraction} for every $x\in \partial \mathbb{A}$ corresponding to the given decomposition. This map will be a boundary retract, but we need to keep the annulus, so we will collapse the ``middle annulus'' only to the ``inner circle'' of it (see definition of $R$ below).
Recall also, that a continuous map between two compact metric spaces is called a \emph{near-homeomorphism}, if it is a uniform limit of homeomorphisms.

The \emph{smash} $R:\mathbb{A}\to \mathbb{A}$ is a near-homeomorphism so that:
$$
R((x,t))=
\begin{cases}
(x,0); & |t|\leq k+2,\\
(x,(k+3)(t-k-2)); & t>k+2,\\
(x,(k+3)(t+k+2)); & t<-k-2.
\end{cases}
$$

 We define the \emph{unwrapping} of $\{G_{\beta}\}_{\beta\in \R}\subset \mathcal{T}$ as a continuously varying family $\bar G_{\beta}\colon \mathbb{A}\to \mathbb{A}$ of orientation-preserving homeomorphisms so that for all $\beta\in \R$:
\begin{enumerate}[(i)]
\item\label{unwrapping:i} if we denote $R\circ \bar{G}_{\beta}(x,t)=(y,s)$ for some $|t|<k+3$ then $|t|>|s|$.  
\item\label{unwrapping:ii} $R\circ \bar{G}_{\beta}|_{[0,1)\times\{0\}}= G_{\beta}$.
\end{enumerate}

Part \eqref{unwrapping:i} is in fact stronger. By our construction, if $t=k+2+r$ then $s=r(k+2)+r$, so under iteration of $R\circ \bar{G}_{\beta}$ each point from strip $[0,1)\times (k+2,k+3)$ has to enter $[0,1)\times [-k-2,k+2]$, and then in next iteration is mapped onto $[0,1)\times\{0\}$. The same argument applies to the strip $[0,1)\times (-k-3,-k-2)$.

We will later show that if $p$ is a fixed point of $G_\beta$ then the radial line emerging from $p$ is invariant under $\bar G_\beta$ and if $p$ is $2$-periodic (i.e. $G_\beta(p)=p+1/2$) then $\bar G_\beta$ interchanges radial lines emerging from $p$ and $p+1/2$.

Set $H_{\beta}:=R\circ \bar{G}_{\beta}$ which is a near-homeomorphism. By Brown's theorem \cite{Br}, $\hat{\mathbb{A}}_{\beta}:=\underleftarrow{\lim}(\mathbb{A},H_{\beta})$ is an annulus as well, thus there exists a homeomorphism $h_{\beta}\colon \hat{\mathbb{A}}_{\beta}\to \mathbb{A}$. Let $\Phi_{\beta}:=h_{\beta}\circ \hat {H}_{\beta}\circ h_{\beta}^{-1}\colon \mathbb{A}\to \mathbb{A}$, where $\hat H_{\beta}$ is a natural extension of $H_{\beta}$ on $\hat{\mathbb{A}}_{\beta}$ and let $\Lambda_{\beta}:=h_{\beta}(\hat S_{\beta})$ where we define $\hat S_\beta:=\underleftarrow{\lim}(\mathbb{S}^1,H_{\beta})$. It follows from \cite{BM} that $\Phi_{\beta}|_{\Lambda_{\beta}}$ is topologically conjugate to $\hat{G}_{\beta}: \underleftarrow{\lim} (\mathbb{S}^1, G_{\beta}) \to \underleftarrow{\lim} (\mathbb{S}^1, G_{\beta})$ for every $\beta\in \R$.  Moreover, it follows from the choice of unwrapping that every point from the interior of $\hat {\mathbb{A}}_{\beta}$ is attracted to $\hat S_{\beta}$, therefore, $\hat S_{\beta}$ is a global attractor for $\hat{H}_{\beta}$ and thus $\Lambda_{\beta}$ is a global attractor for $\Phi_{\beta}$ as well. By Theorem 3.1 from \cite{3G-BM} $\{\Phi_{\beta}\}_{\beta\in \R}$ vary continuously with $\beta\in \R$ and the attractors $\{\Lambda_{\beta}\}_{\beta\in \R}$ vary continuously in the Hausdorff metric.

Now let us consider $\mathbb{A}$ embedded in a topological disk $D\subset \mathbb{R}^2$ such that $\partial D\subset \partial \mathbb{A}$.
To a non-degenerate and non-separating continuum $K\subset D\setminus \partial D$ we can associate the {\em circle of prime ends} $\mathbb{P}$ as the compactification of $D\setminus K$.
If $h\colon\mathbb{R}^2\to \mathbb{R}^2$ preserves orientation and $h(K)=K$, $h(D)=D$ then $h$ induces an orientation preserving homeomorphism $\Omega\colon\mathbb{P}\to\mathbb{P}$, and therefore it gives a natural \emph{prime ends rotation number}. 
In what follows we will also need the following result of Barge~\cite{Barge} where $K$ will denote cofrontiers $\Lambda_{\beta}\subset \mathbb{R}^2$ together with their bounded component. This will make $K$ a plane non-separating continuum. We can naturally extend homeomorphisms $\Phi_{\beta}$ to the continuous family of homeomorphisms of the plane and keep rotations on the boundaries of $\mathbb{A}\subset \mathbb{R}^2$.

\begin{lem}[Proposition 2.2 in \cite{Barge}]\label{Barge}
	Suppose that $\{\Psi_t\}_{t\in[0,1]}$ is a parameterized family of orientation-preserving  homeomorphisms on a topological disk $D\subset \mathbb{R}^2$ continuously varying with $t$. 
	For every $t\in[0,1]$ let $K_t \subset \Int{D}$ be a 
	non-degenerate sphere non-separating continuum, invariant under $\Psi_t$,
	and assume that $\{K_t\}_{t\in[0,1]}$ vary continuously with $t$ in the Hausdorff metric. 
	Then the prime ends rotation numbers vary continuously with $t\in [0,1]$.   	
\end{lem}

Finally, let us define how we distinguish the embeddings from the dynamical perspective. 
In what follows we generalize the definition from \cite{BdCH} of equivalence of embeddings.

\begin{defn}\label{def:equivalent}
	Let $X$ and $Y$ be metric spaces. Suppose that $F:X\to X$ and $G:Y\to Y$ are homeomorphisms and $E:X\to Y$ is an embedding. If $E\circ F=G\circ E$ we say that the embedding $E$ is a {\em dynamical embedding} of $(X,F)$ into $(Y,G)$. If $E$, resp. $E'$, are dynamical embeddings of $(X,F)$ resp. $(X',F')$ into $(Y,G)$, resp. $(Y',G')$, and there is a homeomorphism $H:Y\to Y'$ so that $H(E(X))=E'(X')$ which conjugates $G|_{E(X)}$ with $G'|_{E'(X')}$ we say that the embeddings $E$ and $E'$ are {\em dynamically equivalent}.
\end{defn}
\begin{rem}
	In our case $Y=Y'=\mathbb{R}^2$ and $X,X'$ are the pseudo-circles. Thus, the dynamical equivalence from Definition~\ref{def:equivalent} induces a conjugacy on the ("outside") circles of prime ends without requiring that $H$ conjugates $G$ with $G'$ on all $\mathbb{R}^2$.
\end{rem}

We will also use the following definition and the subsequent basic facts about rotation numbers and rotation sets.

\begin{defn}
	We say that a point $x\in K\subset \mathbb{R}^2$ is {\em accessible} if there exists an arc $A\subset \mathbb{R}^2$ such that $A\cap K=\{x\}$ and $A\setminus \{x\}\subset \mathbb{R}^2\setminus K$.
\end{defn}
 If $f:\mathbb{S}^1\to \mathbb{S}^1$ is a map of degree $1$ and $\tilde f$ is its lift, then we define
$$
\overline{\Rot}_{\tilde f(x)}=\limsup_{m\to\infty}\frac{\tilde f^m(x)-x}{m}.
$$
It is well known that if $\tilde f$ is non-decreasing then $\limsup$ can be replaced by limit and it is independent of the choice of $x$. Define two monotone maps related to the lift $\tilde f$:
$$
\tilde f_l(x)=\inf\{\tilde f(y): y\geq x\}, \quad \tilde f_u(x)=\inf\{\tilde f(y): y\leq x\}.
$$
We define the \textit{rotation set of $\tilde f$} by
$$
\Rot(\tilde f):=\{\overline{\Rot}_{\tilde f}(x) : x\in \R\}.
$$
It is known that for any $x\in \R$ we have:
$$
\Rot(\tilde f)=[\overline{\Rot}_{\tilde f_l}(x),\overline{\Rot}_{\tilde f_u}(x)].
$$
The above definitions can be easily extended onto invariant sets into the annulus. The reader is referred to \cite{ALM} for more deeper treatment of rotation numbers and rotation sets.

Now let us prove the main theorem of this section.
\begin{proof}[Proof of Theorem \ref{lem:BBM1}]
Items (a) and (c) follow directly from Theorem 3.1 of \cite{3G-BM}. The small difference in our setting is that we do not have identity on the boundary of $D$, but it is a circle rotation. This, however, causes no problems since it is not a requirement of Brown's approximation theorem \cite{Br}, see also \cite{CO} for an analogous construction.

Now we will show that each $\Lambda_\beta$ is a Birkhoff-like attractor. Let $f:\Ci\to \Ci$ be the circle map of degree one as on Figure~\ref{fig:mapf} and $\tilde f$ its lift.
 It is clear that there are $n>0$ and $k\in \Z$ such that $\tilde f^n([0,1])\supset [k,k+2]$ and recursively
$\tilde f^{jn}([0,1])\supset [jk,j(k+1)+1]$. But then, there are $x,y\in [0,1]$ such that $\tilde f^{jn}(x)\in [jk,jk+1]$ and $\tilde f^{jn}(y)\in [j(k+1),j(k+1)+1]$
for every positive integer $j$.
This implies that $k/n$ and $(k+1)/n$ belong to $\Rot(\tilde f)$. In particular $\Rot(\tilde f)$ is a non-degenerate interval.
But then, the same property is shared with each $\tilde G_\beta$ for any $\beta$. Following the arguments from Section 4.2 in \cite{3G-BM}, we obtain that $\Rot(\tilde H_\beta)=\Rot(\tilde G_\beta)\cup \{r_{\beta}\}$ where $r_{\beta}$ is the rotation on the $\partial D$,
and $\Rot(\tilde H_\beta|_{\Lambda_\beta})=\Rot(\tilde G_\beta)$, in particular it is a non-degenerate interval. Thus we get that each $\Lambda_\beta$ is indeed a Birkhoff-like attractor.

Now we argue for a proof of (d). Using Theorem~\ref{thm:main} \eqref{main2} and Lemma~\ref{lem:periodic} we can find parameter $\beta_0$ such that $G_{\beta_0}(p)=p$ and a parameter  $\beta_{\frac{1}{2}}$ such that $G^2_{\beta_{\frac{1}{2}}}(p)=p$ and $G_{\beta_{\frac{1}{2}}}(p)\neq p$.  
Note that  $H_{\beta_0}(p,0)=(p,0)$, $H_{\beta_0}(p,k+2)=(p,k+2)$
and $H_{\beta_0}(p,t)\in \{(p,s), s\in [0,k+2]\}$ for each $t\in [0,k+2]$. Denote $A_0:=\{p\}\times [0,k+2]$.
Therefore, $H_{\beta_0}|_{A_0}$ is a well defined monotone surjection on $A_0$ and so $J_{0}:=h_{\beta_0}(\underleftarrow{\lim}(A_0,H_{\beta_0}|_{A_0}))$ defines an arc connecting $P_0:=h_{\beta}(((p,0),(p,0),(p,0),\ldots)))$ with a point of boundary of $D$ and intersects $\Lambda_{\beta_0}$ only at its endpoint $P_0$. Therefore, $P_0$ is an accessible point of $\Lambda_{\beta_0}$.
Let $\Lambda'_\beta$ denote the plane non-separating continua obtained from $\Lambda_{\beta}$ by filling in the inner component of $D\setminus \Lambda_{\beta}$.
Thus, Theorem 5.1 from \cite{Bre} implies that $\mathcal{P}_{0}$ (corresponding to accessible point $P_0$) is a fixed point of the induced homeomorphism $\Omega_{\beta_0}:\mathbb{P}_{\beta_0}\to\mathbb{P}_{\beta_0}$ on the ``outer'' circle of prime ends $\mathbb{P}_{\beta_0}$.
Therefore, it defines a  prime end in $ \mathbb{P}_{\beta_0}$ which is fixed under $\Omega_{\beta_0}$.  Thus, the prime ends rotation number of $\Omega_{\beta_0}$ is $0$.
\\
Now let us argue for $\beta_{\frac{1}{2}}$.
Note that  $H_{\beta_{\frac{1}{2}}}(p,0)=(p+1/2,0)$, $H_{\beta_{\frac{1}{2}}}(p,k+2)=(p+1/2,k+2)$
and $H_{\beta_{\frac{1}{2}}}(p,t)\in \{(p+1/2,s), s\in [0,k+2]\}$ for each $t\in [0,k+2]$. Similarly, $H_{\beta_{\frac{1}{2}}}(p+1/2,t)\in \{(p,s), s\in [0,k+2]\}$ for each $t\in [0,k+2]$.
Denote $A_1:=\{p\}\times [0,k+2]$. Therefore, $H^2_{\beta_{\frac{1}{2}}}|_{A_1}$ is a well defined monotone surjection on $A_1$ and so up to telescoping $J_1:=h_{\beta_{\frac{1}{2}}}(\underleftarrow{\lim}(A_1,H^2_{\beta_{\frac{1}{2}}}|_{A_1}))$ defines an arc connecting $P_1:=h_{\beta_{\frac{1}{2}}}(((p,0),(p+1/2,0),(p,0),(p+1/2,0),\ldots))$ with a point of boundary of $D$ and intersects $\Lambda_{\beta_{\frac{1}{2}}}$ only at its endpoint. Therefore, $P_1$ is an accessible point of $\Lambda_{\beta_{\frac{1}{2}}}$. By a symmetric argument $P'_1:=h_{\beta_{\frac{1}{2}}}(((p+1/2,0),(p,0),(p+1/2,0),(p,0),\ldots))$ is also an accessible point of $\Lambda_{\beta_{\frac{1}{2}}}$.

Now let us show that the rotation number of the induced prime end homeomorphism $\Omega_{\beta_{\frac{1}{2}}}:\mathbb{P}_{\beta_{\frac{1}{2}}}\to\mathbb{P}_{\beta_{\frac{1}{2}}}$ corresponding to $\Phi_{\beta_{\frac{1}{2}}}$ is $1/2$. Similarly as above we argue that there are prime ends $\mathcal{P}_1$ and $\mathcal{P}'_1$ in the corresponding ``outer'' circle of prime ends $\mathbb{P}_{\beta_{\frac{1}{2}}}$. By Theorem 3.2. from \cite{Bre}, if a point from an indecomposable continuum is accessible it corresponds to a unique prime end, thus $\mathcal{P}_1$ and $\mathcal{P}'_1$ are the only prime ends corresponding to accessible points $P_1$ and $P'_1$ respectively. Furthermore, Theorem 5.1 from \cite{Bre} implies that $\Omega^2_{\beta_{\frac{1}{2}}}(\mathcal{P}_1)=\mathcal{P}_1$ and $\Omega^2_{\beta_{\frac{1}{2}}}(\mathcal{P}'_1)=\mathcal{P}'_1$. We only need to exclude that $\Omega_{\beta_{\frac{1}{2}}}(\mathcal{P}_1)=\mathcal{P}_1$ (resp. $\Omega_{\beta_{\frac{1}{2}}}(\mathcal{P}'_1)=\mathcal{P}'_1$). But if $\Omega_{\beta_{\frac{1}{2}}}(\mathcal{P}_1)=\mathcal{P}_1$ (resp. $\Omega_{\frac{1}{2}}(\mathcal{P}'_1)=\mathcal{P}'_1$), the definition of the map $\Omega_{\beta_{\frac{1}{2}}}$ would imply that $P_1$ (resp. $P'_1$) and $\Phi_{\beta_{\frac{1}{2}}}(P_1)$ (resp. $\Phi_{\beta_{\frac{1}{2}}}(P'_1)$) have the same associated equivalence classes of sequences of crosscuts which leads to a contradiction.
This means that the prime ends rotation number associated to the homeomorphism $\Omega_{\beta_{\frac{1}{2}}}$ is $1/2$. Applying Lemma~\ref{Barge} we obtain item (d).

To show item (e) it suffices to use (d) and note that if $\Lambda_t$ and $\Lambda_{t'}$ for $t\neq t'$ are embedded dynamically equivalently, then also the prime end homeomorphisms $\Omega_t$ and  $\Omega_{t'}$ associated to $\Lambda_t$ and $\Lambda_{t'}$ are conjugated (because the associated equivalence classes of sequences of crosscuts are interchanged by the conjugating homeomorphism) which implies the equality of the associated prime ends rotation numbers.

Finally let us argue for (b). For a fixed $\beta$ take a set of generic points $Z_{\beta}\subset [0,1)$ with respect to $\lambda$ of $G_{\beta}$. Let $Z'_{\beta}$ be the image of $Z_{\beta}$ applying rotation by angle $p-G_{\beta}(p)$. Now, if we take set $Y_\beta=Z'_t\times (k+1,k+2)\subset D$ then it has positive two-dimensional Lebesgue measure.
But $H_\beta(Y_\beta)=Z_\beta\times \{0\}$, so it is a set of generic points with respect to $\lambda$ of $H_\beta$ restricted to middle circle of $\mathbb{A}$. So $Y_\beta$
is in the basin of attraction of $\lambda$ on $\mathbb{S}^1$. 

Now we refer to the \cite[Corollary 5.6]{BdCH}, which is an immediate corollary of results of Kennedy, Raines and Stockman \cite{KRS}.

\begin{thm}\label{thm:KRS}
Let $X$ be a compact space with Lebesgue measure $\lambda$, $f:X\to X$  continuous and surjective map and $\nu$ an $f$-invariant Borel probability measure. Then $\nu$ is a physical measure if and only if the induced $\hat f$-invariant measure $\hat{\nu}$ is inverse limit physical on $\underleftarrow{\lim}(X,f)$.
\end{thm}

 By Theorem~\ref{thm:KRS} the basin of attraction of $\lambda$ on $\mathbb{S}^1$ projects onto a physical Oxtoby-Ulam background measure. Since all maps $G_\beta$ preserve Lebesgue measure $\lambda$, the corresponding measures induced on the attractor
change weakly continuously, as a consequence of a more general Theorems~3.6 and 4.1 from \cite{BdCH1}.

Finally, let us show measure-theoretic weak mixing of $G_\beta$ for all $\beta$. By Theorem~4 from \cite{BCOT2} (which is a consequence of results of Bowen~\cite{Bowen}) piecewise $C^2$ leo maps from $C_{\lambda}(I)$ are measure-theoretic exact with respect to $\lambda$, hence
 maps $G_{\beta,\alpha_m,m}$ in our intermediate steps of the construction are measure-theoretic exact with respect to $\lambda$ for each $\beta\in \R$ (and thus also strongly and weakly mixing with respect to $\lambda$). Our aim is to show that if maps $G_{\beta,\alpha_m,m}$ converge sufficiently fast (i.e. $\eta_m$ converges sufficiently fast to zero), then the limit map $G_{\beta}$ is weakly mixing with respect to $\lambda$ for each $\beta\in [0,1]$. The proof is inspired by the proof of Theorem~2 in \cite{BCOT2}.\\
Let $\{h_j\}_{j \ge 1}$ be a countable, dense
collection  of continuous functions in $L^1(\mathbb{S}^1 \times \mathbb{S}^1)$.
Fix any $f\in \{G_{\beta,\alpha_m,m}\}_{\beta\in[0,1]}$ and introduce the following notation for Birkhoff sums
$$S^f_{\ell}h_j(x,y)  := \frac{1}{\ell} \sum_{k=0}^{\ell-1} h_j \big( (f \times f)^k(x,y)\big).$$

This map is weakly mixing, so we have
$$\lim_{\ell \to \infty} S^f_\ell h_j(x,y)  = \int_{\mathbb{S}^1 \times \mathbb{S}^1} h_j(s,t) \, d(\lambda(s) \times \lambda (t))$$
for all $j \ge 1$.

For each $n \ge 1$ there exists a set $E_{n,f} \subset \mathbb{S}^1 \times \mathbb{S}^1$  and a positive integer $\ell_{n,f} \ge n$ such that
$\lambda (E_{n,f}) > 1 - \frac1n$ and
\begin{equation}\label{est''}\Big |S^{f}_{\ell} h_j(x,y) - \int_{\mathbb{S}^1 \times \mathbb{S}^1} h_j(s,t) \,  d(\lambda(s) \times \lambda (t)) \Big | < \frac{1}{3n}
\end{equation}
for all $(x,y) \in E_{n,f}$, $1 \le j \le n$, and $\ell \ge \ell_{n,f}$.

For any $g\in C_{\lambda}(\Ci)$ by the triangular inequality we have:
$$\begin{array}{ll}
 \Big |S^{g}_{\ell_{m, f}} h_j(x,y)   -  \int_{\mathbb{S}^1 \times \mathbb{S}^1} h_j(s,t) \,  d(\lambda(s) \times \lambda (t)) \Big |    \le\\
 \qquad \Big |S^{g}_{\ell_{m, f}} h_j(x,y)  - S^{f}_{\ell_{m, f}} h_j(x,y) \Big | + 
 \Big |S^{f}_{\ell_{m, f}} h_j(x,y)  -   \int_{\mathbb{S}^1 \times \mathbb{S}^1} h_j(s,t) \,  d(\lambda(s) \times
 \lambda (t)) \Big |
 ,\end{array}$$
 so we can choose a small neighbourhood 
$B(f,\eps_{m,f})$ in $C_{\lambda}(\Ci)$ such that 
if $g \in B(f,\eps_{m,f})$ then
\begin{equation}\label{here''}\Big |S^{g}_{\ell_{m,f}} h_j(x,y)  -  \int_{\mathbb{S}^1 \times \mathbb{S}^1} h_j(s,t) \,  d(\lambda(s) \times \lambda (t)) \Big |    < \frac1n
\end{equation}
for $1 \le j\le m$ and all all $(x,y) \in E_{m,f}$.

Since $f=G_{\beta,\alpha_m,m}$ for some $\beta\in [0,1]$, there is some open set $\beta\ni U_f\subset [0,1]$
such that $\{G_{\alpha,\alpha_m,m}:\alpha \in U_f\}\subset B(f,\eps_{m,f}/2)$.
We can find a finite cover of $[0,1]$ by sets $U_{f_1},\ldots, U_{f_{j_n}}$.
Choose $\eps_m := \min \{\eps_{m,f_i}: 1 \le i \le j_m\}$. Then for any $\beta$
and any map $g\in B(G_{\beta,\alpha_m,m},\eps_n)$ condition \eqref{here''} is satisfied with $\ell_{m,f}$ with $f=f_i$ for some $1 \le i \le j_n$.
Denote $\ell_{m,\beta}:=\ell_{m,f}$
and $E_{m,\beta}:=E_{m,f}$
for some $f$ as above (it is not necessarily unique).

Now assume that sequence $\eta_m$ converges to zero sufficiently fast, so that $G_\beta\in B(G_{\beta,\alpha_m,m},\eps_m)$
for each $\beta$.

Consider
$$\mathcal{E}(\beta)  := \bigcap_{M=1}^\infty \bigcup_{m=M}^\infty  E_{m,\beta}.$$
Since $\lambda(E_{m,\beta}) >
1 - \frac{1}{m}$, it follows that  $\lambda(\mathcal{E}(\beta)) = 1$.

We have shown that 

\begin{equation}\label{13}\Big| S^{G_\beta}_{\ell_{m,\beta}} h_j(x,y)  -  \int_{\mathbb{S}^1 \times \mathbb{S}^1} h_j(s,t) \,  d(\lambda(s) \times \lambda (t)) \Big |    < \frac{1}{m}
\end{equation}
for all $(x,y) \in E_{m,\beta}$ and all $1 \le j\le m$.

The
$
\lim_{\ell \to \infty} S^{G_\beta}_{\ell} h_j(x,y)
$
exists for almost every $(x,y)$ by the Birkhoff ergodic theorem. Equation~\eqref{13} holds  infinitely often for every point of $\mathcal{E}(\beta)$  so we conclude that this limit must equal
 $\int_{\mathbb{S}^1 \times \mathbb{S}^1} h_j(s,t) \,  d(\lambda(s) \times \lambda (t))$ and thus
 $G_\beta \times G_\beta$ is ergodic, or equivalently $G_\beta$ is weakly mixing.

\end{proof}

\begin{rem}
Similarly as in \cite{CO}, while the embeddings from Theorem~\ref{lem:BBM1} are different dynamically we can not easily claim that they are different also from the topological point of view. On the other hand, result (d) from Theorem~\ref{lem:BBM1} implies that the parameter space $[0,1]$ is indeed not degenerate. It would be interesting to know how boundary dynamics of the family $\{\Lambda_t\}_{t\in [0,1]}$ looks like precisely (i.e. to understand the sets of accessible points and the prime ends structure), however we again leave this aspect of research as an open problem.
\end{rem}

\section{Acknowledgements}
This research was supported by IDUB program at AGH University.
Research of P. Oprocha was supported in part by National Science Centre, Poland (NCN), grant no. 2019/35/B/ST1/02239. J. \v Cin\v c was partially supported by the Slovenian research agency ARRS grant J1-4632.


\begin{thebibliography}{99}

\bibitem{ALM} L. Alsed\`a, J. Llibre, M. Misiurewicz, Combinatorial dynamics and entropy in dimension one. Second edition. Advanced Series in Nonlinear Dynamics, 5. World Scientific Publishing Co., Inc., River Edge, NJ, 2000. 

\bibitem{AC} A.\ Anu\v si\'c, J.\ \v Cin\v c, \emph{Accessible points of planar embeddings of tent inverse limit spaces}, Diss. Math., {\bf 541}, 57 pp, 2019.

\bibitem{Barge} M.\ Barge, {\em Prime end rotation numbers associated with the Hénon map}, In Continuum Theory and Dynamical Systems Lecture Notes in Pure and Appl. Math. vol. 149 (Dekker, 1993),pp. 15--33.

\bibitem{BM}  M.\ Barge,  J.\ Martin,  {\em The construction of global attractors.}  Proc. Amer. Math. Soc. {\bf 110} (1990), 523--525.

\bibitem{BingPacific} R.\ H.\ Bing, {\em Concerning hereditarily indecomposable continua,}  Pacific J. Math. {\bf 1} (1951), 43--51.

\bibitem{BT} J. Bobok, S. Troubetzkoy,\textit{Typical properties of  interval maps preserving the Lebesgue measure}, Nonlinearity {\bf 33}(12) (2020), 6461--6480.

\bibitem{BCOT1} J. Bobok, J.\ \v Cin\v c, P.\ Oprocha, S. Troubetzkoy, \textit{S-limit shadowing is generic for continuous Lebesgue measure-preserving circle maps}, Ergodic Theory and Dynamical Systems {\bf 43} (1), (2023) 78--98.

\bibitem{BCOT2} J. Bobok, J.\ \v Cin\v c, P.\ Oprocha, S. Troubetzkoy, \textit{Are generic dynamical properties stable under composition with rotations?}, arXiv:2207.07186, July 2022.

\bibitem{BCOT3} J. Bobok, J.\ \v Cin\v c, P.\ Oprocha, S. Troubetzkoy, \textit{ Continuous Lebesgue measure-preserving maps on one-dimensional manifolds: a survey}, arXiv:2303.17873, March 2023.

\bibitem{3G-BM} P.\ Boyland, A.\ de Carvalho, T.\ Hall, {\em Inverse limits as attractors in parametrized families}, Bull. Lond. Math. Soc. {\bf 45}, no. 5 (2013), 1075--1085.

\bibitem{BdCH1} P.\ Boyland, A.\ de Carvalho, T.\ Hall, {\em Statistical Stability for Barge-Martin attractors derived from tent maps}, Discrete and Continuous Dynamical Systems A, {\bf 40} (5) (2020), 2903--2915.

\bibitem{BdCH} P.\ Boyland, A.\ de Carvalho, T.\ Hall, {\em Natural extensions of unimodal maps: prime ends of planar embeddings and semi-conjugacy to sphere homeomorphisms}, Geom. Topol. {\bf 25} (2021) 111--228.


\bibitem{BCL} J. P. Boro\'nski, J. \v Cin\v c, X.-C. Liu, {\em Prime ends dynamics in parametrised families of rotational attractors}, J. London Math. Soc. (2) {\bf 102} (2020) 557--579.
\bibitem{BCO} J. P. Boro\'nski, J. \v Cin\v c, P. Oprocha, {\em Beyond $0$ and $\infty$: A solution to the Barge's entropy conjecture}, arXiv:2105.11133, May 2021. 
\bibitem{BLKO} J. P. Boro\' nski, J. Kennedy, X.-C. Liu, P. Oprocha, {\em Minimal non-invertible maps on the pseudo-circle.} J. Dynam. Differential Equations {\bf 33} (4) (2021), 1897--1916. 
\bibitem{Bowen} R. Bowen, {\em Bernoulli maps of the interval}, Isr. J. Math. {\bf 28} (1977), 161--168.
\bibitem{Bre} B.\ Brechner, {\em On stable homeomorphisms and imbeddings of the pseudo arc}, Illinois J. Math. {\bf 22} (4) (1978), 630--661.
\bibitem{Br} J. R. Brown, \textit{Inverse limits, entropy and weak isomorphism for discrete dynamical systems}, Trans.\ Amer.\ Math.\ Soc.\ {\bf 164}, (1972) 55--66.
\bibitem{Bro} M.\ Brown, {\em Some applications of an approximation theorem for inverse limits}, Proc. Amer. Math. Soc., {\bf 11} (1960) 478--483.

\bibitem{CO} J. \v Cin\v c, P. Oprocha, \textit{Parameterized family of pseudo-arc attractors: physical measures and prime ends rotation number}, Proc. Lond. Math. Soc. {\bf 125} (2023), 318--357.
\bibitem{Fearnley} L. Fearnley, {\em The pseudo-circle is unique}, Tran. Amer. Math. Soc {\bf 149}, (1) (1970), 45--64.
\bibitem{Fearnley2} L. Fearnley, {\em Classification of all hereditarily indecomposable circularly chainable continua.} Trans. Amer. Math. Soc. {\bf 168} (1972), 387--401. 
\bibitem{Handel} M. Handel, {\em A pathological area preserving $C\sp{\infty}$ diffeomorphism of the plane} Proc. Amer. Math. Soc. {\bf 86}(1) (1982), 163--168.

\bibitem{KTT} K. Kawamura, H.M. Tuncali, E.D. Tymchatyn \emph{Hereditarily indecomposable inverse limits of graphs}, Fund. Math. {\bf 185} (2005), 195--210.
\bibitem{KRS} J.\ Kennedy, B.\ E.\ Raines, D.\ R.\ Stockman, \emph{Basins of measures on inverse limit spaces for the induced homeomorphisms}, Ergod.\ Th.\ \& Dynam.\ Sys.\, {\bf 30} (2010), 1119--1130.
\bibitem{KorPas} A. Koropecki, A. Passeggi, {\em A Poincare–Bendixson theorem for translation lines and applications to
prime ends}, Comment. Math. Helv. {\bf 94} (2019) 141--183.

\bibitem{KOT} P. Ko\'scielniak, P. Oprocha, M. Tuncali, \emph{Hereditarily indecomposable inverse limits of graphs: shadowing, mixing and exactness}, Proc. Amer.\ Math.\ Soc.\ {\bf 142}, no. 2 (2014), 681--694. 

\bibitem{LM} W.\ Lewis, P.\ Minc, \emph{Drawing the pseudo-arc}, Houston J. Math. {\bf 36} (2010), 905--934.

\bibitem{MT} P.\ Minc, W.\ R.\ R.\ Transue, \emph{A transitive map on $I$ whose inverse limit is the pseudo-arc}, Proc. Amer.\ Math.\ Soc.\ {\bf 111}, no. 4 (1991), 1165--1170.

\end{thebibliography}
\end{document}